\documentclass[dvips,aos,preprint]{imsart}

\usepackage{amsthm,amsmath,natbib}

\startlocaldefs

\usepackage{amssymb,latexsym}
\usepackage[usenames]{color}
\usepackage[]{graphicx}
\usepackage{mathrsfs}   
\usepackage{caption}
\usepackage{subcaption}
\usepackage{verbatim}
\usepackage{url}
\RequirePackage[colorlinks,citecolor=blue,linkcolor=blue,urlcolor=blue,breaklinks=true]{hyperref}

\theoremstyle{plain}
\newtheorem{theorem}{Theorem}[section]
\newtheorem{corollary}[theorem]{Corollary}
\newtheorem{lemma}[theorem]{Lemma}
\newtheorem{proposition}[theorem]{Proposition}

\newtheorem{condition}[theorem]{Condition}
\newtheorem{conditions}[theorem]{Conditions}
\theoremstyle{definition}
\newtheorem{definition}[theorem]{Definition}
\newtheorem*{unnumbered-definition}{Definition}
\theoremstyle{remark}
\newtheorem*{remark}{Remark}
\numberwithin{equation}{section}

\newcommand{\Capture}{\mathcal{I}}
\newcommand{\apy}{\sigma}

\newcommand{\boldface}{\bf}

\newcommand{\spacing}{\hspace{5mm}}

\newcommand{\eq}[1]{\stackrel{\textup{(#1)}}{=}}
\newcommand{\leql}[1]{\stackrel{\textup{(#1)}}{\leq}}

\newcommand{\iid}{\stackrel{\mathrm{iid}}{\sim}}

\newcommand{\before}{\unlhd}

\newcommand{\I}{\bigcap}
\newcommand{\U}{\bigcup}

\renewcommand{\emptyset}{\varnothing}

\DeclareMathOperator*{\argmax}{argmax}

\DeclareMathOperator*{\Cov}{Cov}
\DeclareMathOperator*{\Var}{Var}
\newcommand{\less}{\smallsetminus}
\renewcommand{\t}{\mathtt{T}}
\DeclareMathOperator*{\support}{support}
\renewcommand{\tilde}{\widetilde}
\renewcommand{\bar}[1]{\mkern 4mu\overline{\mkern-5mu#1\mkern-3mu}\mkern 4mu}

\renewcommand{\a}{\alpha}
\renewcommand{\d}{\delta}
\renewcommand{\epsilon}{\varepsilon}
\renewcommand{\phi}{\varphi}
\newcommand{\new}{\nu}
\newcommand{\m}{\mu}

\newcommand{\A}{\mathcal{A}}

\newcommand{\G}{\mathcal{L}}
\newcommand{\M}{\mathcal{M}}

\renewcommand{\S}{\mathcal{S}}

\newcommand{\X}{\mathcal{X}}
\newcommand{\Y}{\mathcal{Y}}

\DeclareSymbolFont{AMSb}{U}{msb}{m}{n}
\DeclareMathSymbol{\N}{\mathbin}{AMSb}{"4E}
\DeclareMathSymbol{\Z}{\mathbin}{AMSb}{"5A}
\DeclareMathSymbol{\R}{\mathbin}{AMSb}{"52}
\DeclareMathSymbol{\Q}{\mathbin}{AMSb}{"51}
\DeclareMathSymbol{\PP}{\mathbin}{AMSb}{"50}
\renewcommand{\P}{\PP}
\DeclareMathSymbol{\E}{\mathbin}{AMSb}{"45}

\endlocaldefs

\begin{document}

\begin{frontmatter}

\title{Inconsistency of Pitman--Yor process mixtures for the number of components}
\runtitle{Inconsistency}


\begin{aug}
\author{\fnms{Jeffrey W.} \snm{Miller}\corref{}\thanksref{t1}\ead[label=e1]{Jeffrey\_Miller@Brown.edu}}
\and
\author{\fnms{Matthew T.} \snm{Harrison}\thanksref{t1}\ead[label=e2]{Matthew\_Harrison@Brown.edu}}

\thankstext{t1}{Supported in part by NSF grant DMS-1007593 and DARPA contract FA8650-11-1-715.} 

\affiliation{Brown University}

\address{
Division of Applied Mathematics\\
Brown University\\
Providence, RI 02912\\
\printead{e1}\\
\phantom{E-mail:\ }\printead*{e2}}

\runauthor{J. W. Miller and M. T. Harrison}
\end{aug}


\setattribute{journal}{name}{}

\begin{abstract}
In many applications, a finite mixture is a natural model, but it can be difficult to choose an appropriate number of components. To circumvent this choice, investigators are increasingly turning to Dirichlet process mixtures (DPMs), and Pitman--Yor process mixtures (PYMs), more generally. While these models may be well-suited for Bayesian density estimation, many investigators are using them for inferences about the number of components, by considering the posterior on the number of components represented in the observed data. We show that this posterior is not consistent --- that is, on data from a finite mixture, it does not concentrate at the true number of components. This result applies to a large class of nonparametric mixtures, including DPMs and PYMs, over a wide variety of families of component distributions, including essentially all discrete families, as well as continuous exponential families satisfying mild regularity conditions (such as multivariate Gaussians).





\end{abstract}




\end{frontmatter}

\section{Introduction}

\subsection{A motivating example}
\label{section:motivating-example}

In population genetics, determining the ``population structure'' is an important step in the analysis of sampled data. As an illustrative example, consider the impala, a species of antelope in southern Africa. Impalas are divided into two subspecies: the common impala occupying much of the eastern half of the region, and the black-faced impala inhabiting a small area in the west. While common impalas are abundant, the number of black-faced impalas has been decimated by drought, poaching, and declining resources due to human and livestock expansion. To assist conservation efforts, \citet{Lorenzen_2006} collected samples from 216 impalas, and analyzed the genetic variation between/within the two subspecies.

A key part of their analysis consisted of inferring the population structure --- that is, partitioning the data into distinct populations, and in particular, determining how many such populations there are.
To infer the impala population structure, Lorenzen et al.\ employed a widely-used tool called {\sc Structure} \citep{Pritchard_2000} which, in the simplest version, models the data as a finite mixture, with each component in the mixture corresponding to a distinct population. {\sc Structure} uses an ad-hoc method to choose the number of components, but this comes with no guarantees.

Seeking a more principled approach, \citet{Pella_2006} proposed using a Dirichlet process mixture (DPM). Now, in a DPM, the number of components is infinite with probability 1, and thus the posterior on the number of components is always, trivially, a point mass at infinity. Consequently, as is common practice, Pella and Masuda instead employed the posterior on the number of clusters (that is, the number of components used in generating the data observed so far) for inferences about the number of components. 
(The terms ``component'' and ``cluster'' are often used interchangeably, but we make the following crucial distinction: a component is part of a mixture distribution, while a cluster is the set of indices of datapoints coming from a given component.)
This DPM approach was implemented in a software tool called {\sc Structurama} \citep{Huelsenbeck_2007}, and demonstrated on the impala data of Lorenzen et al.; see Figure \ref{figure:posteriors}(a).

\begin{figure}
\centering
\begin{subfigure}{.5\textwidth}
  \centering
  \includegraphics[width=1\linewidth]{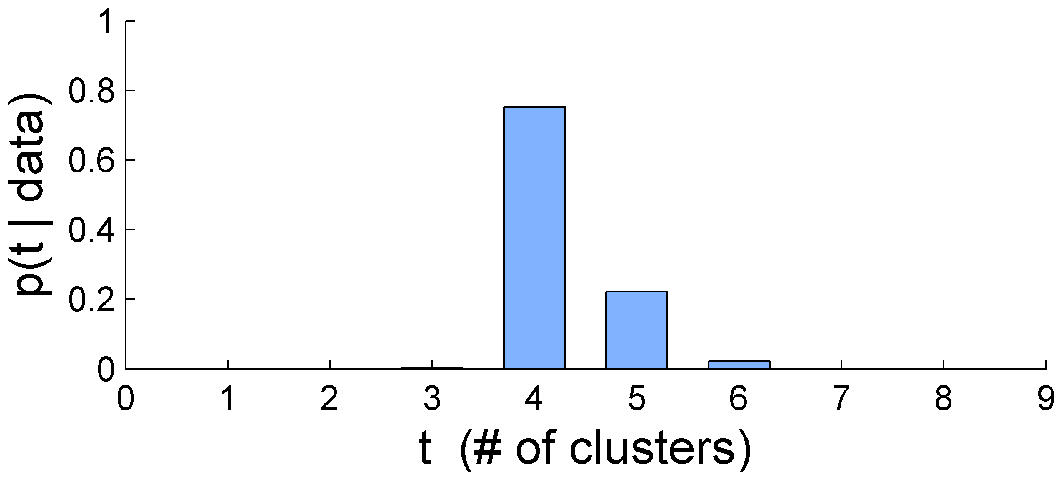}
  \caption{Posterior for impala data}
\end{subfigure}%
\begin{subfigure}{.5\textwidth}
  \centering
  \includegraphics[width=1\linewidth]{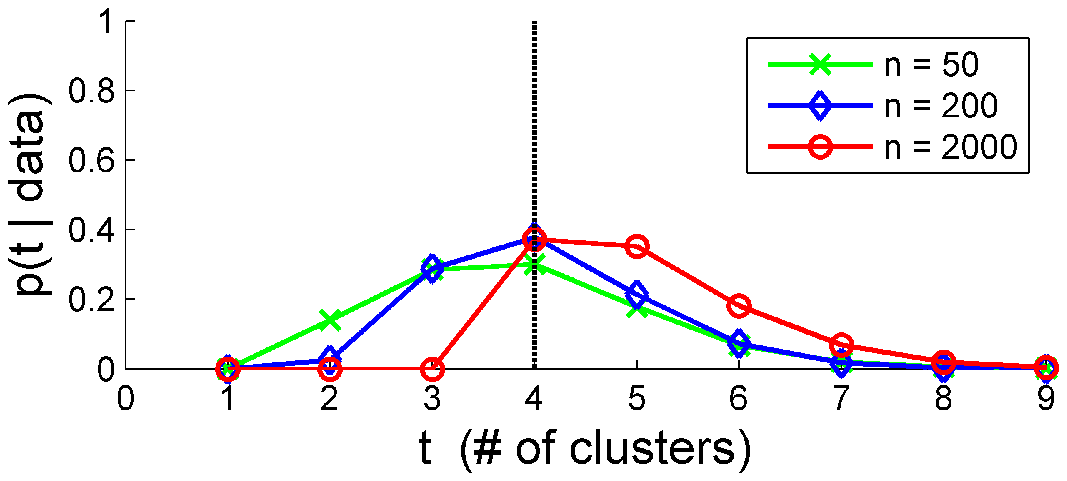}
  \caption{Posterior for Gaussian data}
\end{subfigure}
\caption{Estimated DPM posterior distribution of the number of clusters: (a) For the impala data of Lorenzen et al.\ ($n=216$ datapoints). Our empirical results, shown here, agree with those of Huelsenbeck and Andolfatto. 
(b) For bivariate Gaussian data from a four-component mixture; see Figure \ref{figure:Gaussian_clustering}. Each plot is the average over 10 independently-drawn datasets. (Lines drawn for visualization purposes only.)
(For~(a) and (b), estimates were made via Gibbs sampling, with $10^4$ burn-in sweeps and $10^5$ sample sweeps.)
} 
\label{figure:posteriors}
\end{figure}

{\sc Structurama} has gained acceptance within the population genetics community, and has been used in studies of a variety of organisms, from apples and avocados, to sardines and geckos \citep{Richards_2009,Chen_2009,Gonzalez_2007,Leache_2010}. Studies such as these can carry significant weight, since they may be used by officials to make informed policy decisions regarding agriculture, conservation, and public health.

More generally, in a number of applications the same scenario has played out: a finite mixture seems to be a natural model, but requires the user to choose the number of components, while a Dirichlet process mixture offers a convenient way to avoid this choice. For nonparametric Bayesian density estimation, DPMs are indeed attractive, since the posterior on the density exhibits nice convergence properties; see Section \ref{section:related-work}. However, in several applications, investigators have drawn inferences from the posterior on the number of clusters --- not just the density --- on the assumption that this is informative about the number of components. 
Further examples include
gene expression profiling \citep{Medvedovic_2002},
haplotype inference \citep{Xing_2006},
econometrics \citep{Otranto_2002},
and evaluation of inference algorithms \citep{Fearnhead_2004}.
Of course, if the data-generating process is well-modeled by a DPM (and in particular, there are infinitely many components), then it is sensible to use this posterior for inference about the number of components represented so far in the data --- but that does not seem to be the perspective of these investigators, since they measure performance on simulated data coming from finitely many components or populations.

\begin{figure}[t]
\centering\includegraphics[width=.6\linewidth]{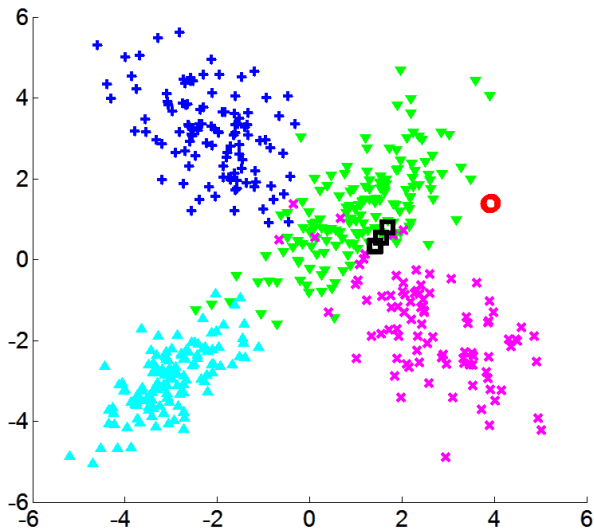}
\caption{A typical partition sampled from the posterior of a Dirichlet process mixture of bivariate Gaussians, on simulated data from a four-component mixture.  Different clusters have different marker shapes ($+,\times,\triangledown,\vartriangle,${\Large$\circ$}$,\Box$) and different colors. Note the tiny ``extra'' clusters ({\Large$\circ$} and $\Box$), in addition to the four dominant clusters.}
\label{figure:Gaussian_clustering}
\end{figure}

Therefore, it is important to understand the properties of this procedure.  Simulation results give some cause for concern; for instance, Figures \ref{figure:posteriors}(b) and \ref{figure:Gaussian_clustering} display results for data from a mixture of two-dimensional Gaussians with four components.  Partitions sampled from the posterior often have tiny ``extra'' clusters, and the posterior on the number of clusters does not appear to be concentrating as the number of datapoints $n$ increases.  This raises a fundamental question that has not been addressed in the literature: With enough data, will this posterior eventually concentrate at the true number of components?  In other words, is it consistent?

\subsection{Overview of results}
\label{section:results}

In this manuscript, we prove that under fairly general conditions, when using a Dirichlet process mixture, the posterior on the number of clusters will not concentrate at any finite value, and therefore will not be consistent for the number of components in a finite mixture. 
In fact, our results apply to a large class of nonparametric mixtures including DPMs, and Pitman--Yor process mixtures (PYMs) more generally, over a wide variety of families of component distributions.

Before treating our general results and their prerequisite technicalities, we would like to highlight a few interesting special cases that can be succinctly stated.  
The terminology and notation used below will be made precise in later sections. 
To reiterate, our results are considerably more general than the following corollary, which is simply presented for the reader's convenience.

\begin{corollary}
\label{corollary:summary}
Consider a Pitman--Yor process mixture with component distributions from one of the following families:
\begin{enumerate}
\item[(a)] $\mathrm{Normal}(\m,\Sigma)$ (multivariate Gaussian),
\item[(b)] $\mathrm{Exponential}(\theta)$,
\item[(c)] $\mathrm{Gamma}(a,b)$,
\item[(d)] $\mathrm{Log\mbox{-}Normal}(\m,\sigma^2)$, or
\item[(e)] $\mathrm{Weibull}(a,b)$ with fixed shape $a>0$,
\end{enumerate}
along with a base measure that is a conjugate prior of the form in Section \ref{section:conjugate-priors}, or
\begin{enumerate}
\item[(f)] any discrete family $\{P_\theta\}$ such that $\I_\theta \{x: P_\theta(x)>0\}\neq\emptyset$ (e.g., Poisson, Geometric, Negative Binomial, Binomial, Multinomial, etc.),
\end{enumerate}
along with any continuous base measure.
Consider any $t\in\{1,2,\dotsc\}$, except for $t = N$ in the case of a Pitman--Yor process with parameters $\apy<0$ and $\vartheta = N|\apy|$.
If $X_1,X_2,\dotsc$ are i.i.d.\ from a mixture with $t$ components from the family used in the model, then the posterior on the number of clusters $T_n$ is not consistent for $t$, and in fact,
$$\limsup_{n\to\infty}\,p(T_n = t\mid X_1,\dotsc,X_n)<1 $$
with probability $1$.
\end{corollary}

This is implied by Theorems \ref{theorem:main}, \ref{theorem:bounded}, and \ref{theorem:exp-case}.
These more general theorems apply to a broad class of partition distributions, handling Pitman--Yor processes as a special case, and they apply to many other families of component distributions: Theorem \ref{theorem:exp-case} covers a large class of exponential families, and Theorem \ref{theorem:bounded} covers families satisfying a certain boundedness condition on the densities (including any case in which the model and data distributions have one or more point masses in common, as well as many location--scale families with scale bounded away from zero). Dirichlet processes are subsumed as a further special case, being Pitman--Yor processes with parameters $\apy = 0$ and $\vartheta>0$.
Also, the assumption of i.i.d.\ data from a finite mixture is much stronger than what is required by these results. 

Regarding the exception of $t = N$ when $\apy<0$ in Corollary \ref{corollary:summary}: posterior consistency at $t = N$ is possible, however, this could only occur if the chosen parameter $N$ just happens to be equal to the actual number of components, $t$.
On the other hand, consistency at any $t$ can (in principle) be obtained by putting a prior on $N$; see Section \ref{section:related-work} below.
In a similar vein, some investigators place a prior on the concentration parameter $\vartheta$ in a DPM, or allow $\vartheta$ to depend on $n$; we conjecture that inconsistency can still occur in these cases, but in this paper, we examine only the case of fixed $\apy$ and $\vartheta$.

\subsection{Discussion / related work}
\label{section:related-work}

We would like to emphasize that this inconsistency should not be viewed as a deficiency of Dirichlet process mixtures, but is simply due to a misapplication of them.
As flexible priors on densities, DPMs are superb, and there are strong results showing that in many cases the posterior on the density converges in $L_1$ to the true density at the minimax-optimal rate, up to a logarithmic factor (see \citet{Scricciolo_2012}, \citet{Ghosal_2010} and references therein).
Further, \citet{Nguyen_2013} has recently shown that the posterior on the mixing distribution converges in the Wasserstein metric to the true mixing distribution.
However, these results do not necessarily imply consistency for the number of components, since any mixture can be approximated arbitrarily well in these metrics by another mixture with a larger number of components (for instance, by making the weights of the extra components infinitesimally small). There seems to be no prior work on consistency of DPMs (or PYMs) for the number of components in a finite mixture (aside from \citet{Miller_2013}, a brief exposition in which we discuss the very special case of a DPM on data from a univariate Gaussian ``mixture'' with one component of known variance).

In the context of ``species sampling'', several authors have studied the Pitman--Yor process posterior (see \citet{Jang_2010,Lijoi_2007} and references therein), but this is very different from our situation --- in a species sampling model, the observed data is drawn directly from a measure with a Pitman--Yor process prior, while in a PYM model, the observed data is drawn from a mixture with such a measure as the mixing distribution.

\citet{Rousseau_2011} proved an interesting result on ``overfitted'' mixtures, in which data from a finite mixture is modeled by a finite mixture with too many components. In cases where this approximates a DPM, their result implies that the posterior weight of the extra components goes to zero. In a rough sense, this is complementary to our results, which involve showing that there are always some nonempty (but perhaps small) extra clusters.

Empirically, many investigators have noticed that the DPM posterior tends to overestimate the number of components (e.g. \citet{West_1994,Lartillot_2004,Onogi_2011}, and others), and such observations are consistent with our theoretical results.
This overestimation seems to occur because there are typically a few tiny ``extra'' clusters. Among researchers using DPMs for clustering, this is an annoyance that is often dealt with by pruning such clusters --- that is, by simply ignoring them when calculating statistics such as the number of clusters. It may be possible to obtain consistent estimators in this way, but this remains an open question; Rousseau and Mengersen's \citeyearpar{Rousseau_2011} results may be applicable here.

However, if one is truly interested in estimating the number of components in a finite mixture, there is no need to resort to such measures --- one can obtain posterior consistency by simply putting a prior on the number of components \citep{Nobile_1994}. (It turns out that putting a prior on $N$ in a PYM with $\apy<0$, $\vartheta= N|\apy|$ is a special case of this \citep{Gnedin_2006}.)
That said, it seems likely that such estimates will be severely affected by misspecification of the model, which is inevitable in most applications. Robustness to model misspecification seems essential for reliable estimation of the number of components, for real-world data.

\subsection{Intuition for the result}
\label{section:intuition}

To illustrate the intuition behind this inconsistency, consider a Dirichlet process with concentration parameter $\vartheta = 1$. (Similar reasoning applies for any Pitman--Yor process with $\apy\geq 0$, but the $\apy<0$ case is somewhat different.)
It is tempting to think that the prior on the number of clusters is the culprit, since (as is well-known) it diverges as $n\to\infty$. Surprisingly, this does not seem to be the main reason why inconsistency occurs.

Instead, the right intuition comes from examining the prior on partitions, \emph{given} the number of clusters.
The prior on ordered partitions $A =(A_1,\dotsc,A_t)$ is $p(A) =(n!\,t!)^{-1}\prod_{i = 1}^t (a_i - 1)!$, where $t$ is the number of parts (i.e. clusters) and $a_i =|A_i|$ is the size of the $i$th part. (The $t!$ comes from uniformly permuting the parts; see Section \ref{section:Gibbs-partitions}.) Since there are $n!/(a_1!\cdots a_t!)$ such partitions with part sizes $(a_1,\dotsc,a_t)$, the conditional distribution of the sizes $(a_1,\dotsc,a_t)$ given $t$ is proportional to $a_1^{-1}\cdots a_t^{-1}$ (subject to the constraint that $\sum a_i = n$). 
See Figure~\ref{figure:corners} for the case of $t = 2$.
The key observation is that, for large $n$, this conditional distribution is heavily concentrated in the ``corners'', where one or more of the $a_i$'s is small. 

\begin{figure}[t]
\centering\includegraphics[width=1\linewidth]{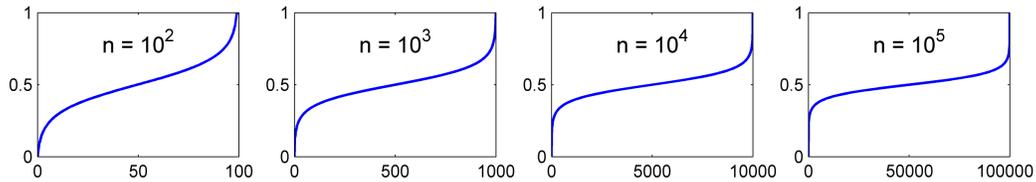}
\caption{The cumulative distribution function of the conditional distribution of $a_1$ given that $t = 2$, for a Dirichlet process with $\vartheta = 1$.  As $n$ increases, the distribution becomes concentrated at the extremes.
}
\label{figure:corners}
\end{figure}

Pursuing this line of thought leads to the following startling fact: the probability of drawing a partition with $t +1$ parts \emph{and one or more of the $a_i$'s equal to $1$} is, at least, the same order of magnitude (with respect to $n$) as the probability of drawing a partition with $t$ parts. 
This leads to the basic idea of the proof --- if the likelihood of the data is also the same order of magnitude, then the posterior probability of $t+1$ will not be too much smaller than that of $t$. Roughly speaking, the posterior will always find it reasonably attractive to split off one element as a singleton.

\subsection{Organization of the paper}
\label{section:organization}

In Section \ref{section:model}, we define Gibbs partition mixture models, which includes Pitman--Yor and Dirichlet process mixtures as special cases.
In Section \ref{section:main}, we prove a general inconsistency theorem for Gibbs partition mixtures satisfying certain conditions.
In Section \ref{section:bounded}, we apply the theorem to cases satisfying a certain boundedness condition on the densities, including discrete families as a special case.
In Section \ref{section:exponential}, we introduce notation for exponential families and conjugate priors, and in Section \ref{section:exp-case}, we apply the theorem to cases in which the mixture is over an exponential family satisfying some regularity conditions.
The remainder of the paper proves the key lemma used in this application.
In Section \ref{section:marginal-inequalities}, we obtain certain inequalities involving the marginal density under an exponential family with conjugate prior.
In Section \ref{section:subset-marginal}, we prove the key lemma of Section \ref{section:exp-case}: an inequality involving the marginal density of any sufficiently large subset of the data.

\section{Model distribution}
\label{section:model}

A primary reason why inconsistency is possible in this situation is that the model is misspecified --- that is, the data comes from a distribution that is not in the model class. 
Thus, our analysis involves two probability distributions: one which is defined by the model, and another which gives rise to the data. In this section, we describe the model distribution.

Dirichlet process mixtures were introduced by \citet{Ferguson_1983} and \citet{Lo_1984} for the purpose of Bayesian density estimation, and were later made practical through the efforts of a number of authors (see \citet{Escobar_1998} and references therein). Pitman--Yor process mixtures \citep{Ishwaran_2003} are a generalization of DPMs based on the Pitman--Yor process \citep{Pitman_1997}.
The model we consider is, in turn, a generalization of PYMs based on the family of Gibbs partitions \citep{Pitman_2006}.

\subsection{Gibbs partitions}
\label{section:Gibbs-partitions}

We will use $p(\cdot)$ to denote probabilities and probability densities under the model.
Our model specification begins with a distribution on partitions, or more precisely, on \emph{ordered} partitions.
Given $n\in\{1,2,\dotsc\}$ and $t\in\{1,\dotsc,n\}$, let $\A_t(n)$ denote the set of all ordered partitions $(A_1,\dotsc,A_t)$ of $\{1,\dotsc,n\}$ into $t$ nonempty sets (or ``parts''). In other words,
$$\A_t(n) =\Big\{(A_1,\dotsc,A_t): A_1,\dotsc,A_t \,\mbox{are disjoint,}\,\,
    \U_{i = 1}^t A_i =\{1,\dotsc,n\},\,\,|A_i|\geq 1\,\,\forall i\Big\}. $$
For each $n\in\{1,2,\dotsc\}$, consider a probability mass function (p.m.f.) on $\U_{t = 1}^n \A_t(n)$ of the form
\begin{align}\label{equation:partition-pmf}
p(A) = v_n(t)\prod_{i = 1}^t w_n(|A_i|)
\end{align}
for $A\in\A_t(n)$, where $v_n:\{1,\dotsc,n\}\to[0,\infty)$ and $w_n:\{1,\dotsc,n\}\to[0,\infty)$. This induces a distribution on $t$ in the natural way, via $p(t\mid A) = I(A\in\A_t(n))$. (Throughout, we use $I$ to denote the indicator function: $I(E)$ is $1$ if $E$ is true, and $0$ otherwise.) It follows that $p(A) = p(A,t)$ when $A\in\A_t(n)$.

Although it is more common to use a distribution on \emph{unordered} partitions $\{A_1,\dotsc,A_t\}$, for our purposes it is more convenient to work with the corresponding distribution on ordered partitions $(A_1,\dotsc,A_t)$ obtained by uniformly permuting the parts.  This does not affect the distribution of $t$.
Under this correspondence, any p.m.f.\ as in Equation \ref{equation:partition-pmf} corresponds to a member of the class of ``exchangeable partition probability functions'', or EPPFs \citep{Pitman_2006}. In particular, for any given $n$ it yields an EPPF in ``Gibbs form'', and a random partition from such an EPPF is called a \emph{Gibbs partition} \citep{Pitman_2006}. (Note: We do not assume that, as $n$ varies, the sequence of p.m.f.s in Equation \ref{equation:partition-pmf} necessarily satisfies the marginalization property referred to as ``consistency in distribution''.)

For example, to obtain the partition distribution for a Dirichlet process (known as a Chinese restaurant process), we can choose
\begin{align}\label{equation:Dirichlet-EPPF}
v_n(t) =\frac{\vartheta^t}{\vartheta_{n\uparrow 1} \,\,t!}\spacing\mbox{ and }\spacing w_n(a) = (a-1)!
\end{align}
where $\vartheta>0$ and $x_{n\uparrow \d} = x(x + \d)(x + 2 \d)\cdots (x +(n-1) \d)$, with $x_{0\uparrow\d} = 1$ by convention. (The $t!$ in the denominator appears since we are working with ordered partitions.) More generally, to obtain the partition distribution for a Pitman--Yor process, we can choose
\begin{align}\label{equation:Pitman--Yor-EPPF}
v_n(t) = \frac{(\vartheta+\apy)_{t-1\uparrow \apy}}{(\vartheta+1)_{n-1\uparrow 1}\,\, t!}
\spacing\mbox{ and }\spacing w_n(a) = (1-\apy)_{a-1\uparrow 1}
\end{align}
where either $\apy\in [0,1)$ and $\vartheta\in(-\apy,\infty)$, or $\apy\in(-\infty,0)$ and $\vartheta=N|\apy|$ for some $N\in\{1,2,\dotsc\}$ \citep{Ishwaran_2003}.
When $\apy = 0$, this reduces to the partition distribution of a Dirichlet process. When $\apy<0$ and $\vartheta= N|\apy|$, it is the partition distribution obtained by drawing $q =(q_1,\dotsc,q_N)$ from a symmetric $N$-dimensional Dirichlet with parameters $|\apy|,\dotsc,|\apy|$, sampling assignments $Z_1,\dotsc,Z_n$ i.i.d.\ from $q$, and removing any empty parts \citep{Gnedin_2006}. Thus, in this latter case, $t$ is always in $\{1,\dotsc,N\}$.

\subsection{Gibbs partition mixtures}
\label{section:Gibbs-partition-mixtures}

Consider the hierarchical model
\begin{align}
& p(A,t) =p(A) = v_n(t)\prod_{i = 1}^t w_n(|A_i|), \label{equation:partitions}\\
& p(\theta_{1:t} \mid A,t) =\prod_{i = 1}^t \pi(\theta_i), \notag\\
& p(x_{1:n} \mid \theta_{1:t},A,t) =\prod_{i = 1}^t \prod_{j\in A_i} p_{\theta_i}(x_j), \notag
\end{align}
where $\pi$ is a prior density on component parameters $\theta\in\Theta\subset\R^k$ for some $k$, and $\{p_\theta:\theta\in\Theta\}$ is a parametrized family of densities on $x\in\X\subset\R^d$ for some~$d$. Here, $x_{1:n} =(x_1,\dotsc,x_n)$ with $x_i\in\X$, $\theta_{1:t} =(\theta_1,\dotsc,\theta_t)$ with $\theta_i\in\Theta$, and $A\in\A_t(n)$.  Assume that $\pi$ is a density with respect to Lebesgue measure, and that $\{p_\theta:\theta\in\Theta\}$ are densities with respect to some sigma-finite Borel measure $\lambda$ on $\X$, such that $(\theta,x)\mapsto p_\theta(x)$ is measurable.  
(Of course, the distribution of $x$ under $p_\theta(x)$ may be discrete, continuous, or neither, depending on the nature of $\lambda$.) 

For $x_1,\dotsc,x_n\in\X$ and $J\subset\{1,\dotsc,n\}$, define the \emph{single-cluster marginal},
\begin{align}\label{equation:marginal}
m(x_J) =\int_\Theta\Big(\prod_{j\in J} p_\theta(x_j)\Big)\,\pi(\theta)\,d\theta,
\end{align}
where $x_J =(x_j: j\in J)$, and assume $m(x_J)<\infty$. By convention, $m(x_J)= 1$ when $J =\emptyset$.
Note that $m(x_J)$ is a density with respect to the product measure $\lambda^\ell$ on $\X^\ell$, where $\ell =|J|$, and that $m(x_J)$ can (and often will) be positive outside the support of $\lambda^\ell$.

\begin{definition}
\label{definition:Gibbs-model}
We refer to such a hierarchical model as a \emph{Gibbs partition mixture model}.
\end{definition}
(Note: This is, perhaps, a slight abuse of the term ``Gibbs partition'', since we allow $v_n$ and $w_n$ to vary arbitrarily with $n$.)
In particular, it is a \emph{Dirichlet process mixture model} when $v_n$ and $w_n$ are as in Equation \ref{equation:Dirichlet-EPPF}, or more generally, a \emph{Pitman--Yor process mixture model} when $v_n$ and $w_n$ are as in Equation \ref{equation:Pitman--Yor-EPPF}. 

We distinguish between the terms ``component'' and ``cluster'': a \emph{component} of a mixture is one of the distributions used in it (e.g. $p_{\theta_i}$), while a \emph{cluster} is the set of indices of datapoints coming from a given component (e.g. $A_i$). 
The prior on the number of clusters under such a model is $p_n(t) = \sum_{A\in\A_t(n)} p(A)$. We use $T_n$, rather than $T$, to denote the random variable representing the number of clusters, as a reminder that its distribution depends on $n$. 

Since we are concerned with the posterior $p(T_n = t \mid x_{1:n})$ on the number of clusters, we will be especially interested in the marginal density of $(x_{1:n},t)$, given~by
\begin{align}
p(x_{1:n},T_n = t) & =\sum_{A\in\A_t(n)}\int p(x_{1:n},\theta_{1:t},A,t)\,d\theta_{1:t} \notag\\
& =\sum_{A\in\A_t(n)} p(A)\prod_{i = 1}^t 
    \int \Big(\prod_{j\in A_i} p_{\theta_i}(x_j)\Big) \pi(\theta_i)\,d\theta_i \notag\\
& =\sum_{A\in\A_t(n)} p(A)\prod_{i = 1}^t m(x_{A_i}). \label{equation:x_t}
\end{align}

As usual, the posterior $p(T_n = t\mid x_{1:n})$ is not uniquely defined, since it can be modified arbitrarily on any subset of $\X^n$ having probability zero under the model distribution. For definiteness, we will employ the usual version of this posterior,
$$ p(T_n = t\mid x_{1:n}) = \frac{p(x_{1:n},T_n = t)}{p(x_{1:n})} =\frac{p(x_{1:n},T_n = t)}{\sum_{t'= 1}^\infty p(x_{1:n},T_n = t')} $$
whenever the denominator is nonzero, and $p(T_n = t\mid x_{1:n}) = 0$ otherwise (for notational convenience).

\section{Inconsistency theorem}
\label{section:main}

The essential ingredients in the main theorem are Conditions \ref{condition:Gibbs} and \ref{condition:phi} below.
For each $n\in\{1,2,\dotsc\}$, consider a partition distribution as in Equation \ref{equation:partition-pmf}.
For $n>t\geq 1$, define
$$c_{w_n} =\max_{a\in\{2,\dotsc,n\}} \frac{w_n(a)}{a\,w_n(a-1)\,w_n(1)}
\spacing\mbox{and}\spacing
c_{v_n}(t) =\frac{v_n(t)}{v_n(t+1)},$$
with the convention that $0/0 = 0$ and $y/0 =\infty$ for $y>0$. 

\begin{condition}
\label{condition:Gibbs}
Assume $\limsup_{n\to\infty} c_{w_n}<\infty$ and $\limsup_{n\to\infty} c_{v_n}(t)<\infty$,
given some particular $t\in\{1,2,\dotsc\}$.
\end{condition}

For Pitman--Yor processes, Condition \ref{condition:Gibbs} holds for all relevant values of $t$, by Proposition \ref{proposition:Gibbs-condition-applies} below.
Now, consider a collection of single-cluster marginals $m(\cdot)$ as in Equation \ref{equation:marginal}. Given $n\geq t\geq 1$, $x_1,\dotsc,x_n\in\X$, and $c\in [0,\infty)$, define
$$\phi_t(x_{1:n},c) = \min_{A\in\A_t(n)}\frac{1}{n}|S_A(x_{1:n},c)|$$
where $S_A(x_{1:n},c)$ is the set of indices $j\in\{1,\dotsc,n\}$ such that the part $A_\ell$ containing $j$ satisfies $m(x_{A_\ell})\leq c\,m(x_{A_\ell\less j})m(x_j)$.

\begin{condition}
\label{condition:phi} Given a sequence of random variables $X_1,X_2,\dotsc\in\X$, a collection of single-cluster marginals $m(\cdot)$, and $t\in\{1,2,\dotsc\}$, assume
$$\sup_{c\in [0,\infty)} \liminf_{n\to\infty}\,\phi_t(X_{1:n},c)>0 \,\mbox{ with probability $1$}.$$
\end{condition}

Note that Condition \ref{condition:Gibbs} involves only the partition distributions, while Condition \ref{condition:phi} involves only the data distribution and the single-cluster marginals. 

\begin{proposition}
\label{proposition:Gibbs-condition-applies}
Consider a Pitman--Yor process. If $\apy\in[0,1)$ and $\vartheta\in(-\apy,\infty)$ then Condition \ref{condition:Gibbs} holds for any $t\in\{1,2,\dotsc\}$. If $\apy\in(-\infty,0)$ and $\vartheta = N|\apy|$, then it holds for any $t\in\{1,2,\dotsc\}$ except $N$.
\end{proposition}
\begin{proof}
This is a simple calculation. See Appendix \ref{section:appendix-proofs}.
\end{proof}

\begin{theorem}
\label{theorem:main}
Let $X_1,X_2,\dotsc\in\X$ be a sequence of random variables (not necessarily i.i.d.).
Consider a Gibbs partition mixture model.
For any $t\in\{1,2,\dotsc\}$, if Conditions \ref{condition:Gibbs} and \ref{condition:phi} hold, then
$$\limsup_{n\to\infty} \,p(T_n = t\mid X_{1:n})<1 \,\mbox{ with probability $1$.}$$
If, further, the sequence $X_1,X_2,\dotsc$ is i.i.d.\ from a mixture with $t$ components, then with probability $1$ the posterior of $T_n$ (under the model) is not consistent for~$t$.
\end{theorem}
\begin{proof}
This follows easily from Lemma \ref{lemma:main} below. See Appendix \ref{section:appendix-proofs}.
\end{proof}

\begin{lemma}
\label{lemma:main}
Consider a Gibbs partition mixture model. Let $n>t\geq 1$, $x_1,\dotsc,x_n\in\X$, and $c\in[0,\infty)$. If $\phi_t(x_{1:n},c)>t/n$, $ c_{w_n}<\infty$, and $c_{v_n}(t)<\infty$, then
$$ p(T_n = t\mid x_{1:n})\leq\frac{C_t(x_{1:n},c)}{1+ C_t(x_{1:n},c)},$$
where $C_t(x_{1:n},c) = t\,c\,c_{w_n}c_{v_n}(t)/(\phi_t(x_{1:n},c)-t/n)$.
\end{lemma}
\begin{proof}
To simplify notation, let us denote $\phi =\phi_t(x_{1:n},c)$, $C = C_t(x_{1:n},c)$, and $S_A = S_A(x_{1:n},c)$ for $A\in\A_t(n)$.
Given $J\subset \{1,\dotsc,n\}$ such that $|J|\geq 1$, define 
$$h_J = w_n(|J|)\,m(x_J). $$
For $A\in\A_t(n)$, let $R_A = S_A\less\big(\U_{i:|A_i|= 1} A_i\big)$, that is, $R_A$ consists of those $j\in S_A$ such that the size of the part $A_\ell$ containing $j$ is greater than $1$. Note that
\begin{align}\label{equation:R_size}
|R_A|\geq|S_A|-t\geq n\phi-t>0.
\end{align}
For any $j\in R_A$, the part $A_\ell$ containing $j$ satisfies 
\begin{align}
h_{A_\ell} &= w_n(|A_\ell|)\,m(x_{A_\ell})\label{equation:h_J-inequality}\\
&\leq c_{w_n}\,|A_\ell|\, w_n(|A_\ell|-1)\,w_n(1)\,c \,m(x_{A_\ell\less j})\,m(x_j)\notag\\
&\leq n \, c\, c_{w_n}\, h_{A_\ell\less j}\, h_j\notag.
\end{align}
Given $j\in R_A$, define $B(A,j)$ to be the element $B$ of $\A_{t +1}(n)$ such that $B_i = A_i\less j$ for $i= 1,\dotsc,t$, and $B_{t+1} =\{j\}$ (that is, remove $j$ from whatever part it belongs to, and make $\{j\}$ the $(t +1)^\textup{th}$ part). Define
$$\Y_A =\big\{B(A,j): j\in R_A\big\}. $$

Now, using Equations \ref{equation:R_size} and \ref{equation:h_J-inequality}, for any $A\in\A_t(n)$ we have
\begin{align}\label{equation:product_h-inequality}
\prod_{i = 1}^t h_{A_i} & =\frac{1}{|R_A|}\sum_{\ell = 1}^t\sum_{j\in R_A\cap A_\ell} h_{A_\ell}\prod_{i\neq \ell} h_{A_i} \\
&\leq\frac{1}{n\phi-t}\sum_{\ell = 1}^t\sum_{j\in R_A\cap A_\ell} n\, c\,c_{w_n}\, h_{A_\ell\less j}\, h_j\prod_{i\neq \ell} h_{A_i}\notag \displaybreak[0]\\
& = \frac{c\,c_{w_n}}{\phi-t/n} \sum_{j \in R_A} \prod_{i = 1}^{t+1} h_{B_i(A,j)} \notag\\
& = \frac{c\,c_{w_n}}{\phi-t/n} \sum_{B\in\A_{t+1}(n)}\Big[\prod_{i = 1}^{t+1} h_{B_i}\Big] I(B\in\Y_A).\notag
\end{align}

For any $B\in\A_{t+1}(n)$,
\begin{align}\label{equation:Y_A-bound}
\#\big\{A\in\A_t(n): B\in\Y_A\big\}\leq t,
\end{align}
since there are only $t$ parts that $B_{t +1}$ could have come from. Therefore,
\begin{align*}
p(x_{1:n},T_n = t) 
& \eq{a} \sum_{A\in\A_t(n)} p(A)\prod_{i = 1}^t m(x_{A_i}) \\
& \eq{b} \sum_{A\in\A_t(n)} v_n(t)\prod_{i = 1}^t h_{A_i} \\
& \leql{c} \frac{c\,c_{w_n}}{\phi-t/n}\, v_n(t) \sum_{A\in\A_t(n)} \sum_{B\in\A_{t+1}(n)}\Big[\prod_{i  = 1}^{t +1} h_{B_i}\Big] I(B\in\Y_A)  \displaybreak[0]\\
& =  \frac{c\,c_{w_n}}{\phi-t/n}\, v_n(t) \sum_{B\in\A_{t+1}(n)} \Big[\prod_{i = 1}^{t +1} h_{B_i}\Big] \#\big\{ A\in\A_t(n): B\in\Y_A\big\}\\
& \leql{d} \frac{c\,c_{w_n}c_{v_n}(t)}{\phi-t/n} \,v_n(t+1) \sum_{B\in\A_{t+1}(n)} \Big[\prod_{i  = 1}^{t +1} h_{B_i}\Big] t\\
& = \frac{t\,c\,c_{w_n}c_{v_n}(t)}{\phi-t/n} \sum_{B\in\A_{t+1}(n)} p(B)\prod_{i = 1}^{t +1} m(x_{B_i})\\
& = C \,\, p(x_{1:n},T_n = t+1),
\end{align*}
where (a) is from Equation \ref{equation:x_t}, (b) is from Equation \ref{equation:partitions} and the definition of $h_J$ above, (c) follows from Equation \ref{equation:product_h-inequality}, and (d) follows from Equation \ref{equation:Y_A-bound}.

If $p(T_n = t\mid x_{1:n}) = 0$, then trivially $p(T_n = t\mid x_{1:n}) \leq C/(C +1)$.
On the other hand, if $p(T_n = t\mid x_{1:n})>0$, then $p(x_{1:n},T_n = t)>0$, and therefore
\begin{flalign*}
&&p(T_n = t\mid x_{1:n}) & =\frac{p(x_{1:n},T_n = t)}{\sum_{t' = 1}^\infty p(x_{1:n},T_n = t')} & \\
&& &\leq \frac{p(x_{1:n},T_n = t)}{p(x_{1:n},T_n = t)+p(x_{1:n},T_n = t+1)} \leq \frac{C}{C +1}. & \qedhere
\end{flalign*}
\end{proof}

\section{Application to discrete or bounded cases}
\label{section:bounded}

By Theorem \ref{theorem:main}, the following result implies inconsistency in a large class of PYM models, including essentially all discrete cases (or more generally anything with at least one point mass) and a number of continuous cases as well.

\begin{theorem}
\label{theorem:bounded}
Consider a family of densities $\{p_\theta:\theta\in\Theta\}$ on $\X$ along with a prior $\pi$ on $\Theta$ and the resulting collection of single-cluster marginals $m(\cdot)$ as in Equation \ref{equation:marginal}. 
Let $X_1,X_2,\dotsc\in\X$ be a sequence of random variables (not necessarily i.i.d.). 
If there exists $U\subset\X$ such that
\begin{enumerate}
\item $\displaystyle\liminf_{n\to\infty}\, \frac{1}{n}\sum_{j = 1}^n I(X_j\in U) > 0$ with probability $1$, and
\item $\displaystyle\sup\Big\{\frac{p_\theta(x)}{m(x)}: x\in U,\,\theta\in\Theta\Big\} <\infty$ (where $0/0 = 0$, $y/0 =\infty$ for $y>0$),
\end{enumerate}
then Condition \ref{condition:phi} holds for all $t\in\{1,2,\dotsc\}$.
\end{theorem}

\begin{proof}
Suppose $U\subset\X$ satisfies (1) and (2), and let $t\in\{1,2,\dotsc\}$. Define $c = \sup\big\{\frac{p_\theta(x)}{m(x)}: x\in U,\,\theta\in\Theta\big\}$. Let $n>t$ and $x_1,\dotsc,x_n\in\X$. Now, for any $x\in U$ and $\theta\in\Theta$, we have $p_\theta(x)\leq c\,m(x)$. Hence, for any $J\subset\{1,\dotsc,n\}$, if $j\in J$ and $x_j\in U$ then
\begin{align}\label{equation:mx_J-inequality}
m(x_J) = \int_\Theta p_\theta(x_j)\Big[\prod_{i\in J\less j} p_\theta(x_i)\Big] \pi(\theta)\,d\theta
\leq c \, m(x_j) m(x_{J\less j}).
\end{align}
Thus, letting $R(x_{1:n}) =\big\{j\in\{1,\dotsc,n\}: x_j\in U\big\}$, we have $R(x_{1:n})\subset S_A(x_{1:n},c)$ for any $A\in\A_t(n)$, and hence, $\phi_t(x_{1:n},c)\geq\frac{1}{n}|R(x_{1:n})|$.

Therefore, by (1), with probability $1$,
$$\liminf_{n\to\infty}\,\phi_t(X_{1:n},c)\geq\liminf_{n\to\infty}\,\frac{1}{n}|R(X_{1:n})|>0. \qedhere$$
\end{proof}

The preceding theorem covers a fairly wide range of cases; here are some examples. Consider a model with $\{p_\theta\}$, $\pi$, $\lambda$, and $m(\cdot)$, as in Section \ref{section:model}.

\begin{enumerate}
\item[(i)] 
{\boldface Finite sample space.}
Suppose $\X$ is a finite set, $\lambda$ is counting measure, and $m(x)>0$ for all $x\in\X$. Then choosing $U =\X$, conditions (1) and~(2) of Theorem \ref{theorem:bounded} are trivially satisfied, regardless of the distribution of $X_1,X_2,\dotsc$.  (Note that when $\lambda$ is counting measure, $p_\theta(x)$ and $m(x)$ are p.m.f.s on $\X$.) It is often easy to check that $m(x)>0$ by using the fact that this is true whenever $\{\theta\in\Theta: p_\theta(x)>0\}$ has nonzero probability under $\pi$.  This case covers, for instance, Multinomials (including Binomials), and the population genetics model from Section \ref{section:motivating-example}.

We should mention a subtle point here: when $\X$ is finite, mixture identifiability might only hold up to a certain maximum number of components (e.g., \citet[Proposition 4]{Teicher_1963} showed this for Binomials), making consistency impossible in general --- however, consistency might still be possible within that identifiable range. Regardless, our result shows that PYMs are not consistent anyway.
\end{enumerate}

Now, suppose $P$ is a probability measure on $\X$, and $X_1,X_2,\dotsc\iid P$.  Let us abuse notation and write $P(x)=P(\{x\})$ and $\lambda(x) =\lambda(\{x\})$ for $x\in\X$.

\begin{enumerate}
\item[(ii)]
{\boldface One or more point masses in common.}
If there exists $x_0 \in\X$ such that $P(x_0)>0$, $\lambda(x_0)>0$, and $m(x_0)>0$, then it is easy to verify that conditions (1) and (2) are satisfied with $U =\{x_0\}$.
(Note that $\lambda(x_0)>0$ implies $p_\theta(x_0)\leq 1/\lambda(x_0)$ for any $\theta\in\Theta$.)

\item[(iii)]
{\boldface Discrete families.}
Case (ii) essentially covers all discrete families --- e.g., Poisson, Geometric, Negative Binomial, or any power-series distribution (see \citet{Sapatinas_1995} for mixture identifiability of these) --- provided that the data is i.i.d.. 
For, suppose $\X$ is a countable set and $\lambda$ is counting measure. By case (ii), the theorem applies if there is any $x_0\in\X$ such that $m(x_0)>0$ and $P(x_0)>0$. If this is not so, the model is extremely misspecified, since then the model distribution and the data distribution are mutually singular.

\item[(iv)]
{\boldface Continuous densities bounded on some non-null compact set.}
Suppose there exists $c\in(0,\infty)$ and $U\subset\X$ compact such that
\begin{enumerate}
\item $P(U)>0$,
\item $x\mapsto p_\theta(x)$ is continuous on $U$ for all $\theta\in\Theta$, and
\item $p_\theta(x) \in(0,c]$ for all $x\in U$, $\theta\in\Theta$.
\end{enumerate}
Then condition (1) is satisfied due to item (a), and condition (2) follows easily from (b) and (c) since $m(x)$ is continuous (by the dominated convergence theorem) and positive on the compact set $U$, so $\inf_{x\in U} m(x)>0$.
This case covers, for example, the following families (with any $P$):
\begin{enumerate}
\item $\mathrm{Exponential}(\theta)$, $\X=(0,\infty)$,
\item $\mathrm{Gamma}(a,b)$, $\X=(0,\infty)$, with variance $a/b^2$ bounded away from zero,
\item $\mathrm{Normal}(\m,\Sigma)$, $\X =\R^d$, (multivariate Gaussian) with $\det(\Sigma)$ bounded away from zero, and
\item many location--scale families with scale bounded away from zero (for instance, $\mathrm{Laplace}(\m,\sigma)$ or $\mathrm{Cauchy}(\m,\sigma)$, with $\sigma\geq\epsilon>0$).
\end{enumerate}

\end{enumerate}

The examples listed in item (iv) are indicative of a deficiency in Theorem \ref{theorem:bounded}: condition (2) is not satisfied in some important cases, such as multivariate Gaussians with unrestricted covariance. Showing that Condition \ref{condition:phi} still holds, for many exponential families at least, is the objective of the remainder of the paper.

\section{Exponential families and conjugate priors}
\label{section:exponential}

\subsection{Exponential families}
\label{section:exponential-families}

In this section, we make the usual definitions for exponential families and state the regularity conditions to be assumed.
Consider an exponential family of the following form. Fix a sigma-finite Borel measure $\lambda$ on $\X\subset\R^d$ such that $\lambda(\X)\neq 0$, let $s:\X\to\R^k$ be Borel measurable, and for $\theta\in\Theta\subset\R^k$, define a density $p_\theta$ with respect to $\lambda$ by setting
$$ p_\theta(x) = \exp(\theta^\t s(x)-\kappa(\theta)) $$
where
$$\kappa(\theta) =\log\int_{\X} \exp(\theta^\t  s(x)) \,d\lambda(x). $$
Let $P_\theta$ be the probability measure on $\X$ corresponding to $p_\theta$, that is,
$P_\theta(E) =\int_E p_\theta(x) \,d\lambda(x)$
for $E\subset\X$ measurable. 
Any exponential family on $\R^d$ can be written in the form above by reparametrizing if necessary, and choosing $\lambda$ appropriately. 
We will assume the following (very mild) regularity conditions.

\begin{conditions}
\label{conditions:regularity}
Assume the family $\{P_\theta:\theta\in\Theta\}$ is:
\begin{enumerate}
\item full, that is, $\Theta =\{\theta\in\R^k:\kappa(\theta)<\infty\}$,
\item nonempty, that is, $\Theta \neq\emptyset$,
\item regular, that is, $\Theta$ is an open subset of $\R^k$, and
\item identifiable, that is, if $\theta\neq\theta'$ then $P_\theta\neq P_{\theta'}$.
\end{enumerate}
\end{conditions}

Most commonly-used exponential families satisfy Conditions \ref{conditions:regularity}, including multivariate Gaussian, Gamma, Poisson, Exponential, Geometric, and others. 
(A notable exception is the Inverse Gaussian, for which $\Theta$ is not open.)
Let $\M$ denote the \emph{moment space}, that is,
$$\M =\{\E_\theta s(X):\theta\in\Theta\} $$
where $\E_\theta$ denotes expectation under $P_\theta$.
Finiteness of these expectations is guaranteed, thus $\M\subset\R^k$; see Appendix \ref{section:exponential-properties} for this and other well-known properties that we will use.

\subsection{Conjugate priors}
\label{section:conjugate-priors}

Given an exponential family $\{P_\theta\}$ as above, let 
$$\Xi=\Big\{(\xi,\new):\xi\in\R^k,\,\new>0 \mbox{ s.t. } \xi/\new\in \M\Big\}, $$
and consider the family $\{\pi_{\xi,\new} : (\xi,\new)\in \Xi\}$ where
$$ \pi_{\xi,\new}(\theta) 
   = \exp\big(\xi^\t\theta -\new\kappa(\theta)-\psi(\xi,\new)\big) \, I(\theta\in\Theta) $$
is a density with respect to Lebesgue measure on $\R^k$.
Here,
$$\psi(\xi,\new) = \log\int_\Theta \exp\left(\xi^\t\theta-\new \kappa(\theta)\right) d\theta. $$
In Appendix \ref{section:exponential-properties}, we note a few basic properties of this family --- in particular, it is a conjugate prior for $\{P_\theta\}$.

\begin{definition}
\label{definition:well-behaved}
We will say that an exponential family with conjugate prior is \emph{well-behaved} if it takes the form above, satisfies Conditions \ref{conditions:regularity}, and has $(\xi,\new)\in\Xi$.
\end{definition}

\section{Application to exponential families}
\label{section:exp-case}

In this section, we apply Theorem~\ref{theorem:main} to prove that in many cases, a PYM model using a well-behaved exponential family with conjugate prior will exhibit inconsistency for the number of components.

\begin{conditions}
\label{conditions:data}
Consider an exponential family with sufficient statistics function $s:\X\to\R^k$ and moment space $\M$.
Given a probability measure $P$ on $\X$, let $X\sim P$ and assume:
\begin{enumerate}
\item $\E|s(X)|<\infty$,
\item $\P(s(X)\in\bar\M) = 1$, and
\item $\P(s(X)\in L) = 0$ for any hyperplane $L$ that does not intersect $\M$.
\end{enumerate}
\end{conditions}
Throughout, we use $|\cdot|$ to denote the Euclidean norm.
Here, a \emph{hyperplane} refers to a set $L =\{x\in\R^k: x^\t y = b\}$ for some $y\in\R^k\less\{0\}$, $b\in\R$. 
In Theorem~\ref{theorem:exp-case} below, it is assumed that the data comes from a distribution $P$ satisfying Conditions \ref{conditions:data}. In Proposition \ref{proposition:continuous-data}, we give some simple conditions under which, if $P$ is a finite mixture from the exponential family under consideration, then Conditions \ref{conditions:data} hold. 

The following theorem follows almost immediately from Lemma \ref{lemma:subset-marginal-bound}, the proof of which will occupy most of the remainder of the paper.

\begin{theorem}
\label{theorem:exp-case}
Consider a well-behaved exponential family with conjugate prior (as in Definition \ref{definition:well-behaved}),  along with the resulting collection of single-cluster marginals $m(\cdot)$.
Let $P$ be a probability measure on $\X$ satisfying Conditions \ref{conditions:data} (for the $s$ and $\M$ from the exponential family under consideration), and let $X_1,X_2,\dotsc\iid P$.  
Then Condition \ref{condition:phi} holds for any $t\in\{1,2,\dotsc\}$.
\end{theorem}

\begin{proof}
Let $t\in\{1,2,\dotsc\}$ and choose $c$ according to Lemma \ref{lemma:subset-marginal-bound} with $\beta = 1/t$. 
We will show that for any $n>t$, if the event of Lemma \ref{lemma:subset-marginal-bound} holds, then $\phi_t(X_{1:n},c)\geq 1/(2 t)$. Since with probability $1$, this event holds for all $n$ sufficiently large, it will follow that with probability $1$, $\,\liminf_n \phi_t(X_{1:n},c)\geq 1/(2 t)>0$.

So, let $n > t$ and $x_1,\dotsc,x_n\in\X$, and assume the event of Lemma \ref{lemma:subset-marginal-bound} holds.
Let $A\in\A_t(n)$. There is at least one part $A_\ell$ such that $|A_\ell|\geq n/t=\beta n$. Then, by assumption there exists $R_A\subset A_\ell$ such that $|R_A|\geq\frac{1}{2} |A_\ell|$ and for any $j\in R_A$, $m(x_{A_\ell})\leq c\,m(x_{A_\ell\less j})\,m(x_j)$. Thus, $R_A\subset S_A(x_{1:n},c)$, hence $|S_A(x_{1:n},c)|\geq|R_A|\geq\frac{1}{2}|A_\ell|\geq n/(2 t)$. Since $A\in\A_t(n)$ was arbitrary, $\phi_t(x_{1:n},c)\geq 1/(2 t)$.
\end{proof}

This theorem implies inconsistency in several important cases. In particular, it can be verified that each of the following is well-behaved (when put in canonical form and given the conjugate prior in Section \ref{section:conjugate-priors}) and, using Proposition \ref{proposition:continuous-data} below, that if $P$ is a finite mixture from the same family then $P$ satisfies Conditions \ref{conditions:data}:
\begin{enumerate}
\item[(a)] $\mathrm{Normal}(\m,\Sigma)$ (multivariate Gaussian),
\item[(b)] $\mathrm{Exponential}(\theta)$,
\item[(c)] $\mathrm{Gamma}(a,b)$,
\item[(d)] $\mathrm{Log\mbox{-}Normal}(\m,\sigma^2)$, and
\item[(e)] $\mathrm{Weibull}(a,b)$ with fixed shape $a>0$.
\end{enumerate}
Combined with the cases covered by Theorem \ref{theorem:bounded}, these results are fairly comprehensive.

\begin{proposition}
\label{proposition:continuous-data}
Consider an exponential family $\{P_\theta:\theta\in\Theta\}$ satisfying Conditions \ref{conditions:regularity}. If $X\sim P =\sum_{i = 1}^t \pi_i P_{\theta(i)}$ for some $\theta(1),\dotsc,\theta(t)\in\Theta$ and some $\pi_1,\dotsc,\pi_t\geq 0$ such that $\sum_{i = 1}^t \pi_i = 1$, then
\begin{enumerate}
\item $\E|s(X)|<\infty$, and
\item $\P(s(X)\in\bar\M) = 1$.
\end{enumerate}
If, further, the exponential family is continuous (that is, the underlying measure $\lambda$ is absolutely continuous with respect to Lebesgue measure on $\X$), $\X\subset\R^d$ is open and connected, and the sufficient statistics function $s:\X\to\R^k$ is real analytic (that is, each coordinate function $s_1,\dotsc,s_k$ is real analytic), then
\begin{enumerate}
\item[\textup{(3)}] $\P(s(X)\in L) = 0$ for \emph{any} hyperplane $L\subset\R^k$.
\end{enumerate}
\end{proposition}
\begin{proof}
This is relatively straightforward; see the \ref{supplementary-material}.
\end{proof}

Sometimes, Condition \ref{conditions:data}(3) will be satisfied even when Proposition \ref{proposition:continuous-data} is not applicable. In any particular case, it may be a simple matter to check this condition by using the characterization of $\M$ as the interior of the closed convex hull of $\support(\lambda s^{-1})$ (see Proposition \ref{proposition:properties}(8) in the Appendix).

\section{Marginal inequalities}
\label{section:marginal-inequalities}

Consider a well-behaved exponential family with conjugate prior (as in Definition \ref{definition:well-behaved}). In this section, we use some simple bounds on the Laplace approximation (see Appendix \ref{section:Laplace-bounds}) to prove certain inequalities involving the marginal density (from Equation \ref{equation:marginal}),
$$ m(x_{1:n}) =\int_\Theta \Big(\prod_{j = 1}^n p_\theta(x_j)\Big) \pi_{\xi,\new}(\theta)\,d\theta $$
of $x_{1:n} =(x_1,\dotsc,x_n)$, where $x_j\in\X$. 
Of course, it is commonplace to apply the Laplace approximation to $m(X_{1:n})$ when $X_1,\dotsc,X_n$ are i.i.d.\ random variables. In contrast, our application of it is considerably more subtle. For our purposes, it is necessary to show that the approximation is good not only in the i.i.d.\ case, but in fact whenever the sufficient statistics are not too extreme.

We make extensive use of the exponential family properties in Appendix \ref{section:exponential-properties}, often without mention.
We use $f'$ to denote the gradient and $f''$ to denote the Hessian of a (sufficiently smooth) function $f:\R^k\to\R$.
For $\m\in \M$, define
\begin{align*}
& f_\m(\theta) = \theta^\t\m-\kappa(\theta), \\
& \G(\m) =\sup_{\theta\in\Theta} \big(\theta^\t\m-\kappa(\theta)\big), \\
& \theta_\m = \argmax_{\theta\in\Theta} \big(\theta^\t\m-\kappa(\theta)\big),
\end{align*}
and note that $\theta_\m =\kappa'^{-1}(\m)$ (Proposition \ref{proposition:properties}). $\G$ is known as the Legendre transform of $\kappa$. Note that $\G(\m) = f_\m(\theta_\m)$, and $\G$ is $C^\infty$ smooth on $\M$ (since $\G(\m) =\theta_\m^\t\m-\kappa(\theta_\m)$, $\,\,\theta_\m =\kappa'^{-1}(\m)$, and both $\kappa$ and $\kappa'^{-1}$ are $C^\infty$ smooth).
Define 
\begin{align}
\label{equation:m_x}
\m_{x_{1:n}} =\frac{\xi+\sum_{j = 1}^n s(x_j)}{\new + n}
\end{align}
(cf. Equation \ref{equation:joint-exponential}), and given $x_{1:n}$ such that $\m_{x_{1:n}}\in \M$, define
$$\tilde m(x_{1:n}) = (\new+n)^{-k/2} \exp\big((\new+n)\,\G(\m_{x_{1:n}})\big),$$
where $k$ is the dimension of the sufficient statistics function $s:\X\to\R^k$.
The first of the two results of this section provides uniform bounds on $m(x_{1:n})/\tilde m(x_{1:n})$. 
Here, $\tilde m(x_{1:n})$ is only intended to approximate $m(x_{1:n})$ up to a multiplicative constant --- a better approximation could always be obtained via the usual asymptotic form of the Laplace approximation.

\begin{proposition}
\label{proposition:marginal-approximation}
Consider a well-behaved exponential family with conjugate prior.
For any $U\subset \M$ compact, there exist $C_1,C_2 \in(0,\infty)$ such that for any $n\in\{1,2,\dotsc\}$ and any $x_1,\dotsc,x_n\in\X$ satisfying $\m_{x_{1:n}} \in U$, we have
$$ C_1\leq\frac{m(x_{1:n})}{\tilde m(x_{1:n})}\leq C_2. $$
\end{proposition}
\begin{proof}
Assume $U\neq\emptyset$, since otherwise the result is trivial.
Let 
$$V =\kappa'^{-1}(U) =\{\theta_\m:\m\in U\}. $$
It is straightforward to show that there exists $\epsilon\in(0,1)$ such that $V_\epsilon\subset\Theta$ where
$$ V_\epsilon =\{\theta\in\R^k: d(\theta,V)\leq\epsilon\}. $$
(Here, $d(\theta,V) =\inf_{\theta'\in V} |\theta-\theta'|$.) Note that $V_\epsilon$ is compact, since $\kappa'^{-1}$ is continuous. Given a symmetric matrix $A$, define $\lambda_*(A)$ and $\lambda^*(A)$ to be the minimal and maximal eigenvalues, respectively, and recall that $\lambda_*,\lambda^*$ are continuous functions of the entries of $A$. Letting 
$$\a =\min_{\theta\in V_\epsilon} \lambda_*(\kappa''(\theta))
\spacing\mbox{and}\spacing
\beta =\max_{\theta\in V_\epsilon} \lambda^*(\kappa''(\theta)),$$
we have $0<\a\leq\beta<\infty$ since $V_\epsilon$ is compact and $\lambda_*(\kappa''(\cdot))$, $\lambda^*(\kappa''(\cdot))$ are continuous and positive on $\Theta$.
Letting
$$\gamma =\sup_{\m\in U} e^{-f_\m(\theta_\m)} \int_\Theta \exp(f_\m(\theta)) d\theta 
=\sup_{\m\in U} e^{-\G(\m)} e^{\psi(\m,1)}$$
we have $0<\gamma<\infty$ since $U$ is compact, and both $\G$ (as noted above) and $\psi(\m,1)$ (by Proposition \ref{proposition:conjugate}) are continuous on $\M$. Define
$$ h(\m,\theta) = f_\m(\theta_\m)-f_\m(\theta) = \G(\m)-\theta^\t \m +\kappa(\theta) $$
for $\m\in \M$, $\theta\in\Theta$. For any $\m\in \M$, we have that $h(\m,\theta)>0$ whenever $\theta\in\Theta\less\{\theta_\m\}$, and that $h(\m,\theta)$ is strictly convex in $\theta$. Letting $B_\epsilon(\theta_\m) =\{\theta\in\R^k:|\theta-\theta_\m|\leq\epsilon\}$, it follows that
$$\d:=\inf_{\m\in U}\inf_{\theta\in\Theta\less B_\epsilon(\theta_\m)} h(\m,\theta)
     =\inf_{\m\in U}\inf_{u\in\R^k:|u|= 1} h(\m,\theta_\m +\epsilon u)$$
is positive, as the minimum of a positive continuous function on a compact set.

Now, applying the Laplace approximation bounds in Corollary \ref{corollary:Laplace-bounds} with $\a,\beta,\gamma,\d,\epsilon$ as just defined, we obtain $c_1,c_2 \in(0,\infty)$ such that for any $\m\in U$ we have (taking $E =\Theta$, $f =-f_\m$, $x_0 =\theta_\m$, $A =\a I_{k\times k}$, $B =\beta I_{k\times k}$)
$$ c_1\leq\frac{\int_\Theta \exp(t f_\m(\theta)) d\theta}{t^{-k/2}\exp(t f_\m(\theta_\m))}\leq c_2 $$
for any $t\geq 1$.
We prove the result with $C_i = c_i\,e^{-\psi(\xi,\new)}$ for $i = 1,2$.

Let $n\in\{1,2,\dotsc\}$ and $x_1,\dotsc,x_n\in\X$ such that $\m_{x_{1:n}}\in U$. Choose $t = \new+n$.  By integrating Equation \ref{equation:joint-exponential}, we have
$$ m(x_{1:n}) = e^{-\psi(\xi,\new)} \int_\Theta \exp\big(t f_{\m_{x_{1:n}}}(\theta)\big)\,d\theta, $$
and meanwhile,
$$\tilde m(x_{1:n})=t^{-k/2}\exp\big(t f_{\m_{x_{1:n}}}(\theta_{\m_{x_{1:n}}})\big).$$
Thus, combining the preceding three displayed equations,
$$0<C_1 = c_1 e^{-\psi(\xi,\new)} 
   \leq\frac{m(x_{1:n})}{\tilde m(x_{1:n})}\leq c_2 e^{-\psi(\xi,\new)} = C_2<\infty. \qedhere$$
\end{proof}

The second result of this section is an inequality involving a product of marginals.

\begin{proposition}[Splitting inequality]
\label{proposition:splitting-inequality}
Consider a well-behaved exponential family with conjugate prior.
For any $U\subset \M$ compact there exists $C\in(0,\infty)$ such that we have the following:

For any $n\in\{1,2,\dotsc\}$, if $A\subset\{1,\dotsc,n\}$ and $B =\{1,\dotsc,n\}\less A$ are nonempty, and $x_1,\dotsc,x_n\in\X$ satisfy $\frac{1}{|A|}\sum_{j\in A} s(x_j)\in U$ and $\m_{x_B}\in U$, then
$$\frac{m(x_{1:n})}{m(x_A) m(x_B)}\leq C\,\,\Big(\frac{a b}{\new+n}\Big)^{k/2} $$
where $a =\new +|A|$ and $b =\new +|B|$.
\end{proposition}
\begin{proof}
Let $U'$ be the convex hull of $U\cup \{\xi/\new\}$. Then $U'$ is compact (as the convex hull of a compact set in $\R^k$) and $U'\subset\M$ (since $U\cup\{\xi/\new\}\subset\M$ and $\M$ is convex). We show that the result holds with $C = C_2 \exp(C_0)/C_1^2$, where $C_1,C_2\in(0,\infty)$ are obtained by applying Proposition \ref{proposition:marginal-approximation} to $U'$, and
\begin{align}
\label{equation:C_0}
C_0 = \new \sup_{y\in U'}|(\xi/\new-y)^\t \G'(y)|+\new\sup_{y\in U'}|\G(y)|<\infty.
\end{align}

Since $\G$ is convex (being a Legendre transform) and smooth, then for any ${y,z\in \M}$ we have
$$\inf_{\rho\in(0,1)} \frac{1}{\rho}\big(\G(y +\rho(z-y))-\G(y)\big) = (z-y)^\t \G'(y) $$
(by e.g. \cite{Rockafellar_1970} 23.1) and therefore for any $\rho\in(0,1)$,
\begin{align}
\label{equation:Legendre-inequality}
\G(y)\leq \G((1-\rho)y + \rho z)-\rho(z-y)^\t \G'(y).
\end{align}
Choosing $y =\m_{x_{1:n}}$, $z =\xi/\new$, and $\rho =\new/(n +2\new)$, we have
\begin{align}
\label{equation:convex-combination}
(1-\rho) y + \rho z =\frac{2\xi + \sum_{j = 1}^n s(x_j)}{2\new+n} = \frac{a \m_{x_A}+ b\m_{x_B}}{a+b}.
\end{align}
Note that $\m_{x_A},\m_{x_B},\m_{x_{1:n}}\in U'$, by taking various convex combinations of $\xi/\new$, $\frac{1}{|A|}\sum_{j\in A} s(x_j)$, $\m_{x_B} \in U'$. Thus,
\begin{align*}
(\new+n) \G(\m_{x_{1:n}})&= (a+b)\G(y)-\new \G(y) \\
&\leql{a} (a+b)\G((1-\rho) y + \rho z)-(a+b)\rho(z-y)^\t \G'(y)-\new \G(y)  \displaybreak[0]\\
&\leql{b} (a+b) \G\Big(\frac{a \m_{x_A}+ b\m_{x_B}}{a+b}\Big) + C_0 \\
&\leql{c} a \G(\m_{x_A}) + b \G(\m_{x_B}) + C_0,
\end{align*}
where (a) is by Equation \ref{equation:Legendre-inequality}, (b) is by Equations \ref{equation:C_0} and \ref{equation:convex-combination}, and (c) is by the convexity of $\G$.
Hence, $(\new+n)^{k/2} \tilde m(x_{1:n})\leq (ab)^{k/2}\tilde m(x_A)\tilde m(x_B)\exp(C_0)$,
so by our choice of $C_1$ and $C_2$,
$$\frac{m(x_{1:n})}{m(x_A) m(x_B)}\leq \frac{C_2 \tilde m(x_{1:n})}{C_1^2\tilde m(x_A)\tilde m(x_B)}
\leq\frac{C_2\exp(C_0)}{C_1^2}\Big(\frac{ab}{n +\new}\Big)^{k/2}. \qedhere$$
\end{proof}

\section{Marginal inequality for subsets of the data}
\label{section:subset-marginal}

In this section, we prove Lemma~\ref{lemma:subset-marginal-bound}, the key lemma used in the proof of Theorem \ref{theorem:exp-case}.
First, we need a few supporting results.

Given $y_1,\dotsc,y_n\in\R^\ell$ (for some $\ell>0$), $\beta\in(0,1]$, and $U\subset\R^\ell$, define
$$ \Capture_\beta(y_{1:n},U) =\prod_{\substack{A\subset\{1,\dotsc,n\}:\\|A|\geq\beta n}} I\Big(\frac{1}{|A|}\sum_{j\in A} y_j\in U\Big), $$
where as usual, $I(E)$ is $1$ if $E$ is true, and $0$ otherwise.

\begin{lemma}[Capture lemma]
\label{lemma:capture}
Let $V\subset\R^k$ be open and convex. Let $Q$ be a probability measure on $\R^k$ such that:
\begin{enumerate}
\item $\E|Y|<\infty$ when $Y\sim Q$,
\item $Q(\bar V) = 1$, and
\item $Q(L) = 0$ for any hyperplane $L$ that does not intersect $V$.
\end{enumerate}
If $Y_1,Y_2,\dotsc\iid Q$, then for any $\beta\in(0,1]$ there exists $U\subset V$ compact such that $\Capture_\beta(Y_{1:n},U)\xrightarrow[]{\mathrm{a.s.}} 1$ as $n\to\infty$.
\end{lemma}
\begin{proof}
The proof is rather long, but not terribly difficult. For details, see the \ref{supplementary-material}.
\end{proof}

\begin{proposition}
\label{proposition:singleton-helper}
Let $Z_1,Z_2,\dotsc\in\R^k$ be i.i.d.. If $\beta\in(0,1]$ and $U\subset\R^k$ such that $\P(Z_j\not\in U)<\beta/2$, then $\Capture_\beta(Y_{1:n},[\tfrac{1}{2},1])\xrightarrow[]{\mathrm{a.s.}} 1$ as $n \to\infty$, where $Y_j = {I(Z_j\in U)}$.
\end{proposition}
\begin{proof}
By the law of large numbers, $\frac{1}{n}\sum_{j = 1}^n I(Z_j\not\in U)\xrightarrow[]{\textup{a.s.}} \P(Z_j\not\in U) < \beta/2$. Hence, with probability $1$, for all $n$ sufficiently large, $\frac{1}{n}\sum_{j = 1}^n I(Z_j\not\in U)\leq\beta/2$ holds. When it holds, we have that for any $A\subset\{1,\dotsc,n\}$ such that $|A|\geq\beta n$,
$$\frac{1}{|A|}\sum_{j\in A} I(Z_j\in U) = 1-\frac{1}{|A|}\sum_{j\in A} I(Z_j\not\in U) \geq 1-\frac{1}{\beta n}\sum_{j = 1}^n I(Z_j\not\in U)\geq 1/2,$$
i.e. when it holds, we have $\Capture_\beta(Y_{1:n},[\tfrac{1}{2},1])= 1$.
Hence, $\Capture_\beta(Y_{1:n},[\tfrac{1}{2},1])\xrightarrow[]{\mathrm{a.s.}} 1$.
\end{proof}

In the following, $\m_x = (\xi + s(x))/(\new+1)$, as in Equation \ref{equation:m_x}.

\begin{proposition}
\label{proposition:singleton-capture}
Consider a well-behaved exponential family with conjugate prior. Let $P$ be a probability measure on $\X$ such that $\P(s(X)\in\bar\M) = 1$ when $X\sim P$. Let $X_1,X_2,\dotsc\iid P$. Then for any $\beta\in(0,1]$ there exists $U\subset\M$ compact such that $\Capture_\beta(Y_{1:n},[\tfrac{1}{2},1])\xrightarrow[]{\mathrm{a.s.}} 1$ as $n\to\infty$, where $Y_j = I(\m_{X_j}\in U)$.
\end{proposition}
\begin{proof}
Since $\M$ is open and convex, then for any $y\in\bar\M$, $z\in\M$, and ${\rho\in(0,1)}$, we have $\rho y +(1-\rho) z\in\M$ (by e.g. \citet{Rockafellar_1970} 6.1). Taking $z =\xi/\new$ and $\rho = 1/(\new+1)$, this implies that the set $U_0 =\{(\xi+y)/(\new+1) : y\in\bar\M\}$ is contained in $\M$. Note that $U_0$ is closed and $\P(\m_X\in U_0) = \P(s(X)\in\bar\M) = 1$. Let $\beta\in(0,1]$, and choose $r\in(0,\infty)$ such that $\P(|\m_X|> r)<\beta/2$. Letting $U =\{y\in U_0:|y|\leq r\}$, we have that $U\subset\M$, and $U$ is compact. Further, $\P(\m_X\not\in U)<\beta/2$, so by applying Proposition \ref{proposition:singleton-helper} with $Z_j =\m_{X_j}$, we have $\Capture_\beta(Y_{1:n},[\tfrac{1}{2},1])\xrightarrow[]{\mathrm{a.s.}} 1$.
\end{proof}

\begin{lemma}
\label{lemma:subset-marginal-bound}
Consider a well-behaved exponential family with conjugate prior, and the resulting collection of single-cluster marginals $m(\cdot)$. Let $P$ be a probability measure on $\X$ satisfying Conditions \ref{conditions:data} (for the $s$ and $\M$ from the exponential family under consideration), and let $X_1,X_2,\dotsc\iid P$. Then for any $\beta\in(0,1]$ there exists $c\in(0,\infty)$ such that with probability $1$, for all $n$ sufficiently large, the following event holds: for every subset $J\subset\{1,\dotsc,n\}$ such that $|J|\geq\beta n$, there exists $K\subset J$ such that $|K|\geq\frac{1}{2}|J|$ and for any $j\in K$,
$$ m(X_J)\leq c\, m(X_{J\less j})\,m(X_j). $$
\end{lemma}
\begin{proof}
Let $\beta\in (0,1]$. Since $\M$ is open and convex, and Conditions \ref{conditions:data} hold by assumption, then by Lemma \ref{lemma:capture} (with $V =\M$) there exists $U_1\subset\M$ compact such that $\Capture_{\beta/2}(s(X_{1:n}),U_1)\xrightarrow[]{\mathrm{a.s.}} 1$ as $n\to\infty$, where $s(X_{1:n}) =(s(X_1),\dotsc,s(X_n))$. By Proposition \ref{proposition:singleton-capture} above, there exists $U_2\subset\M$ compact such that ${\Capture_\beta(Y_{1:n},[\tfrac{1}{2},1])\xrightarrow[]{\mathrm{a.s.}} 1}$ as $n\to\infty$, where $Y_j = I(\m_{X_j}\in U_2)$. Hence,
$$\Capture_{\beta/2}(s(X_{1:n}),U_1)\,\Capture_\beta(Y_{1:n},[\tfrac{1}{2},1])\xrightarrow[n\to\infty]{\mathrm{a.s.}} 1. $$
Choose $C\in(0,\infty)$ according to Proposition \ref{proposition:splitting-inequality} applied to $U: = U_1\cup U_2$. We will prove the result with $c= (\new+1)^{k/2} C$. (Recall that $k$ is the dimension of $s:\X\to\R^k$.)

Let $n$ large enough that $\beta n\geq 2$, and suppose that $\Capture_{\beta/2}(s(X_{1:n}),U_1)= 1$ and $\Capture_\beta(Y_{1:n},[\tfrac{1}{2},1])= 1$. Let $J\subset\{1,\dotsc,n\}$ such that $|J|\geq\beta n$. Then for any $j\in J$,
$$\frac{1}{|J\less j|}\sum_{i\in J\less j} s(X_i)\in U_1\subset U$$
since $\Capture_{\beta/2}(s(X_{1:n}),U_1)= 1$ and $|J\less j|\geq |J|/2\geq (\beta /2) n$.
Hence, for any $j\in K$, where $K =\{j\in J:\m_{X_j}\in U\}$, we have
$$\frac{m(X_J)}{m(X_{J\less j})\,m(X_j)}\leq C\,\,\Big(\frac{(\new+|J|-1)(\new+1)}{\new+|J|}\Big)^{k/2}\leq C\,(\new+1)^{k/2} = c$$
by our choice of $C$ above, and
$$\frac{|K|}{|J|}\geq\frac{1}{|J|}\sum_{j\in J} I(\m_{X_j}\in U_2) =\frac{1}{|J|}\sum_{j\in J} Y_j\geq 1/2$$
since $\Capture_\beta(Y_{1:n},[\tfrac{1}{2},1])= 1$ and $|J|\geq\beta n$.
\end{proof}

\appendix

\section{}
\label{section:appendix-proofs}

\begin{proof}[Proof of Proposition \ref{proposition:Gibbs-condition-applies}]
There are two cases: (A) $\apy\in[0,1)$ and $\vartheta>-\apy$, or (B) $\apy<0$ and $\vartheta = N|\apy|$. In either case, $\apy<1$, so
$$\frac{w_n(a)}{a w_n(a-1) w_n(1)} = \frac{1-\apy +a -2}{a} \leq\frac{1-\apy}{2} +1$$
whenever $n\geq 2$ and $a\in\{2,\dotsc,n\}$, and hence $\limsup_n c_{w_n}<\infty$.

For any $n>t\geq 1$, in case (A) we have
$$\frac{v_n(t)}{v_n(t+1)}=\frac{t+1}{\vartheta + t \apy},$$
and the same holds in case (B) if also $t<N$. Meanwhile, whenever $N<t<n$ in case (B), $v_n(t)/v_n(t+1)= 0/0 = 0$ by convention. Therefore, $\limsup_n c_{v_n}(t)<\infty$ in either case, for any $t\in\{1,2,\dotsc\}$ except $t = N$ in case (B).
\end{proof}

\begin{proof}[Proof of Theorem \ref{theorem:main}]
Let $t\in\{1,2,\dotsc\}$, and assume Conditions \ref{condition:Gibbs} and \ref{condition:phi} hold. Let $x_1,x_2,\dotsc\in\X$, and suppose $\sup_{c \in [0,\infty)} \liminf_n \phi_t(x_{1:n},c)>0$ (which occurs with probability $1$). We show that this implies $\limsup_n p(T_n = t\mid x_{1:n})<1$, proving the theorem.

Let $\a\in(0,\infty)$ such that $\limsup_n c_{w_n}<\a$ and $\limsup_n c_{v_n}(t)<\a$.
Choose $c \in [0,\infty)$ and $\epsilon\in(0,1)$ such that $\liminf_n \phi_t(x_{1:n},c)>\epsilon$.
Choose $N>2 t/\epsilon$ large enough that for any $n>N$ we have $c_{w_n}<\a$, $c_{v_n}(t)<\a$, and $\phi_t(x_{1:n},c)>\epsilon$.
Then by Lemma \ref{lemma:main}, for any $n>N$,
$$ p(T_n = t\mid x_{1:n})\leq\frac{C_t(x_{1:n},c)}{1+ C_t(x_{1:n},c)}
\leq\frac{2 t c \a^2/\epsilon}{1+2 t c \a^2/\epsilon},$$
since $\phi_t(x_{1:n},c)-t/n>\epsilon-\epsilon/2 =\epsilon/2$ (and $y\mapsto y/(1+ y)$ is monotone increasing on $[0,\infty)$). Since this upper bound does not depend on $n$ (and is less than $1$), then $\limsup_n p(T_n = t\mid x_{1:n})<1$.
\end{proof}

\section{Exponential family properties}
\label{section:exponential-properties}

We note some well-known properties of exponential families satisfying Conditions \ref{conditions:regularity}.
For a general reference on this material, see \citet{Hoffmann_1994}.
Let $S_\lambda(s) = \support(\lambda s^{-1})$, that is,
$$ S_\lambda(s) =\big\{z\in\R^k:\lambda(s^{-1}(U)) \neq 0 
    \,\mbox{ for every neighborhood }\,U \mbox{ of } z\big\}. $$
Let $C_\lambda(s)$ be the closed convex hull of $S_\lambda(s)$ (that is, the intersection of all closed convex sets containing it).
Given $U\subset\R^k$, let $ U^\circ$ denote its interior.
Given a (sufficiently smooth) function $f:\R^k\to\R$, we use $f'$ to denote its gradient, that is,
$f'(x)_i=\frac{\partial f}{\partial x_i}(x)$,
and $f''(x)$ to denote its Hessian matrix, that is,
$f''(x)_{ij} = \frac{\partial^2 f}{\partial x_i \partial x_j} (x)$.

\begin{proposition}
\label{proposition:properties}
If Conditions \ref{conditions:regularity} are satisfied, then:
\begin{enumerate}
\item $\kappa$ is $C^\infty$ smooth and strictly convex on $\Theta$,
\item $\kappa'(\theta) =\E s(X)$ and $\kappa''(\theta) =\Cov s(X)$ when $\theta\in\Theta$ and $X\sim P_\theta$,
\item $\kappa''(\theta)$ is symmetric positive definite for all $\theta\in\Theta$,
\item $\kappa':\Theta\to\M$ is a $C^\infty$ smooth bijection,
\item $\kappa'^{-1}: \M\to\Theta$ is $C^\infty$ smooth,
\item $\Theta$ is open and convex,
\item $\M$ is open and convex,
\item $\M = C_\lambda(s)^\circ$ and $\bar \M = C_\lambda(s)$, and
\item $\kappa'^{-1}(\mu) =\argmax_{\theta\in\Theta} (\theta^\t\mu-\kappa(\theta))$ for all $\mu\in \M$. The maximizing $\theta\in\Theta$ always exists and is unique.
\end{enumerate}
\end{proposition}
\begin{proof}
These properties are all well-known. Let us abbreviate \citet{Hoffmann_1994} as HJ. For (1), see HJ 8.36(1) and HJ 12.7.5. For (6),(2),(3), and (4), see HJ 8.36, 8.36.2-3, 12.7(2), and 12.7.11, respectively. Item (5) and openness in (7) follow, using the inverse function theorem \citep[3.21]{Knapp_2005}. Item (8) and convexity in (7) follow, using HJ 8.36.15 and \citet{Rockafellar_1970} 6.2-3. Item (9) follows from HJ 8.36.15 and item (4).
\end{proof}

Given an exponential family with conjugate prior as in Section \ref{section:conjugate-priors}, the joint density of $x_1,\dotsc,x_n\in\X$ and $\theta\in\R^k$ is
\begin{align}\label{equation:joint-exponential}
& p_\theta(x_1) \cdots p_\theta(x_n) \pi_{\xi,\new}(\theta)\\
& =\exp\Big((\new+n)\big(\theta^\t\m_{x_{1:n}}-\kappa(\theta)\big)\Big) \exp(-\psi(\xi,\new)) \,I(\theta\in\Theta) \notag
\end{align}
where $\m_{x_{1:n}} = (\xi + \sum_{j = 1}^n s(x_j))/(\new + n)$.
The marginal density, defined as in Equation \ref{equation:marginal}, is
\begin{align}
\label{equation:exponential-marginal}
m(x_{1:n}) =\exp\Big(\psi\big(\xi +{\textstyle\sum s(x_j)}, \,\new + n\big)-\psi(\xi,\new)\Big)
\end{align}
when this quantity is well-defined.

\begin{proposition}
\label{proposition:conjugate}
If Conditions \ref{conditions:regularity} are satisfied, then:
\begin{enumerate}
\item $\psi(\xi,\new)$ is finite and $C^\infty$ smooth on $\Xi$,
\item if $s(x_1),\dotsc,s(x_n)\in S_\lambda(s)$ and $(\xi,\new)\in\Xi$, then $(\xi+\sum s(x_j),\, \new + n)\in\Xi$,
\item $\{\pi_{\xi,\new} : (\xi,\new)\in \Xi\}$ is a conjugate family for $\{p_\theta:\theta\in\Theta\}$, and
\item if $s:\X\to\R^k$ is continuous, $(\xi,\new)\in\Xi$, and $\lambda(U)\neq 0$ for any nonempty $U\subset\X$ that is open in $\X$, then $m(x_{1:n})<\infty$ for any $x_1,\dotsc,x_n\in\X$.
\end{enumerate}
\end{proposition}
\begin{proof}
(1) For finiteness, see \citet{Diaconis_1979}, Theorem 1. 
Smoothness holds for the same reason that $\kappa$ is smooth \citep[8.36(1)]{Hoffmann_1994}. (Note that $\Xi$ is open in $\R^{k +1}$, since $\M$ is open in $\R^k$.)

(2) Since $C_\lambda(s)$ is convex, $\frac{1}{n}\sum s(x_j)\in C_\lambda(s)$. Since $C_\lambda(s) =\bar\M$ and $\M$ is open and convex (\ref{proposition:properties}(7) and (8)), then $(\xi+\sum s(x_j))/(\new + n)\in\M$, as a (strict) convex combination of $\frac{1}{n}\sum s(x_j)\in\bar\M$ and $\xi/\new\in\M$ \citep[6.1]{Rockafellar_1970}.

(3) Let $(\xi,\new)\in\Xi$, $\theta\in\Theta$. If $X_1,\dotsc,X_n\iid P_\theta$ then $s(X_1),\dotsc,s(X_n)\in S_\lambda(s)$ almost surely, and thus $(\xi+\sum s(X_j),\, \new + n)\in\Xi$ (a.s.) by (2). By Equations \ref{equation:joint-exponential} and \ref{equation:exponential-marginal}, the posterior is $\pi_{\xi+\sum s(X_j),\, \new + n}$.

(4) The assumptions imply $\{s(x): x\in\X\}\subset S_\lambda(s)$, and therefore, for any $x_1,\dotsc,x_n\in\X$, we have $(\xi+\sum s(x_j),\, \new + n)\in\Xi$ by (2). Thus, by (1) and Equation \ref{equation:exponential-marginal}, $m(x_{1:n})<\infty$.
\end{proof}

It is worth mentioning that while $\Xi\subset\big\{(\xi,\new)\in\R^{k+1} : \psi(\xi,\new)<\infty\big\}$, it may be a strict subset --- often, $\Xi$ is not quite the full set of parameters on which $\pi_{\xi,\new}$ can be defined. 


\section{Bounds on the Laplace approximation}
\label{section:Laplace-bounds}

Our proof uses the following simple bounds on the Laplace approximation. These bounds are not fundamentally new, but the precise formulation we require does not seem to appear in the literature, so we have included it for the reader's convenience. Lemma \ref{lemma:Laplace-bounds} is simply a multivariate version of the bounds given by \citet{deBruijn_1970}, and Corollary \ref{corollary:Laplace-bounds} is a straightforward consequence, putting the lemma in a form most convenient for our purposes.

Given symmetric matrices $A$ and $B$, let us write $A\before B$ to mean that $B-A$ is positive semidefinite. Given $A\in\R^{k\times k}$ symmetric positive definite and $\epsilon,t\in(0,\infty)$, define
$$ C(t,\epsilon,A) =\P(|A^{-1/2} Z|\leq\epsilon\sqrt{t}) $$
where $Z\sim\mathrm{Normal}(0,I_{k\times k})$. Note that $C(t,\epsilon,A)\to 1$ as $t\to\infty$. Let
$B_\epsilon(x_0) =\{x\in\R^k:|x-x_0|\leq\epsilon\}$
denote the closed ball of radius $\epsilon>0$ at $x_0\in\R^k$.

\begin{lemma}
\label{lemma:Laplace-bounds}
Let $E\subset\R^k$ be open. Let $f:E\to\R$ be $C^2$ smooth with $f'(x_0) = 0$ for some $x_0\in E$. Define
$$ g(t) =\int_E \exp(-t f(x))\,dx $$
for $t\in(0,\infty)$. Suppose $\epsilon\in(0,\infty)$ such that $B_\epsilon(x_0)\subset E$, $0<\d\leq\inf\{f(x)-f(x_0): x\in E\less B_\epsilon(x_0)\}$, and $A,B$ are symmetric positive definite matrices such that $A\before f''(x)\before B$ for all $x\in B_\epsilon(x_0)$. Then for any $0<s\leq t$ we have
$$\frac{C(t,\epsilon,B)}{|B|^{1/2}}
\leq\frac{g(t)}{(2\pi/t)^{k/2} e^{-t f(x_0)}}
\leq\frac{C(t,\epsilon,A)}{|A|^{1/2}} + \Big(\frac{t}{2\pi}\Big)^{k/2} e^{-(t-s)\d} e^{s f(x_0)} g(s)$$
where $|A|=|\det A|$.
\end{lemma}
\begin{remark}
In particular, these assumptions imply $f$ is strictly convex on $B_\epsilon(x_0)$ with unique global minimum at $x_0$. Note that the upper bound is trivial unless $g(s)<\infty$.
\end{remark}
\begin{proof}
This is a straightforward application of Taylor's theorem; see the \ref{supplementary-material}.
\end{proof}

The following corollary tailors the lemma to our purposes. Given a symmetric positive definite matrix $A\in\R^{k\times k}$, let $\lambda_*(A)$ and $\lambda^*(A)$ be the minimal and maximal eigenvalues, respectively. By diagonalizing $A$, it is easy to check that $\lambda_*(A) I_{k\times k}\before A\before\lambda^*(A) I_{k\times k}$ and $\lambda_*(A)^k\leq|A|\leq\lambda^*(A)^k$.

\begin{corollary}
\label{corollary:Laplace-bounds}
For any $\a,\beta,\gamma,\d,\epsilon\in(0,\infty)$ there exist $c_1 = c_1(\beta,\epsilon)\in(0,\infty)$ and $c_2 = c_2(\a,\gamma,\d)\in(0,\infty)$ such that if $E,f,x_0,A,B$ satisfy all the conditions of Lemma \ref{lemma:Laplace-bounds} (for this choice of $\d,\epsilon$) and additionally, $\a\leq\lambda_*(A)$, $\beta\geq\lambda^*(B)$, and $\gamma\geq e^{f(x_0)} g(1)$, then
$$c_1\leq\frac{\int_E \exp(-t f(x)) \,dx}{t^{-k/2}\exp(-t f(x_0))} \leq c_2 $$
for all $t\geq 1$.
\end{corollary}
\begin{proof}
The first term in the upper bound of the lemma is $C(t,\epsilon,A)/|A|^{1/2}\leq 1/\a^{k/2}$, and with $s = 1$ the second term is less or equal to $(t/2\pi)^{k/2} e^{-(t -1)\d} \gamma$, which is bounded above for $t\in [1,\infty)$. For the lower bound, a straightforward calculation (using $z^\t Bz\leq \lambda^*(B) z^\t z \leq \beta z^\t z$ in the exponent inside the integral) shows that $C(t,\epsilon,B)/|B|^{1/2}\geq \P(|Z|\leq\epsilon\sqrt{\beta})/\beta^{k/2}$ for $t\geq 1$.
\end{proof}


\section*{Acknowledgments}
We would like to thank Stu Geman for raising this question. We would also like to thank Erik Sudderth, Mike Hughes, Tamara Broderick, Annalisa Cerquetti, Steve MacEachern, Dan Klein, and Dahlia Nadkarni for helpful conversations.

\begin{supplement}[id=supplementary-material]
\sname{Supplementary Material}
\stitle{Supplement to ``Inconsistency of Pitman--Yor process mixtures for the number of components''}
\slink[doi]{TBD}
\sdatatype{.pdf}
\sdescription{We have placed some technical proofs in a supplementary document, \citet{supplementary}.}
\end{supplement}

\bibliography{refs}
\bibliographystyle{imsart-nameyear}

\end{document}


\begin{frontmatter}

\title{Supplement to ``Inconsistency of Pitman--Yor process mixtures for the number of components''}
\runtitle{Supplementary Material}


\begin{aug}
\author{\fnms{Jeffrey W.} \snm{Miller}\corref{}\thanksref{t1s}\ead[label=e1s]{Jeffrey\_Miller@Brown.edu}}
\and
\author{\fnms{Matthew T.} \snm{Harrison}\thanksref{t1s}\ead[label=e2s]{Matthew\_Harrison@Brown.edu}}

\thankstext{t1s}{Supported in part by NSF grant DMS-1007593 and DARPA contract FA8650-11-1-715.} 

\affiliation{Brown University}

\address{
Division of Applied Mathematics\\
Brown University\\
Providence, RI 02912\\
\printead{e1s}\\
\phantom{E-mail:\ }\printead*{e2s}}

\runauthor{J. W. Miller and M. T. Harrison}
\end{aug}


\setattribute{journal}{name}{}




\end{frontmatter}

This supplementary document contains various proofs that were excluded from the main document \citep{main_document} in the interest of space.


\section{Miscellaneous proofs}

\begin{proof}[Proof of Proposition 6.3]
(1) For any $\theta\in\Theta$ and any $j\in\{1,\dotsc,k\}$,
$$\int_{\X} s_j(x)^2 p_\theta(x)\,d\lambda(x) = \exp(-\kappa(\theta))
    \frac{\partial^2}{\partial\theta_j^2} \int_{\X} \exp(\theta^\t s(x)) \,d\lambda(x) <\infty $$
\citep[8.36.1]{Hoffmann_1994}. Since $P$ has density $f =\sum \pi_i p_{\theta(i)}$ with respect to $\lambda$, then
$$\E s_j(X)^2 =\int_\X s_j(x)^2 f(x)\,d\lambda(x) =\sum_{i = 1}^t \pi_i\int_\X s_j(x)^2 p_{\theta(i)}(x)\,d\lambda(x)<\infty,$$
and hence
$$(\E|s(X)|)^2\leq\E|s(X)|^2 =\E s_1(X)^2+\cdots +\E s_k(X)^2 <\infty.$$

(2) Note that $S_P(s)\subset S_\lambda(s)$ (in fact, they are equal since $P_\theta$ and $\lambda$ are mutually absolutely continuous for any $\theta\in\Theta$), and therefore
$$S_P(s)\subset S_\lambda(s) \subset C_\lambda(s) = \bar \M$$
by Proposition B.1(8). Hence, 
$$\P(s(X)\in\bar \M) \geq\P(s(X)\in S_P(s)) = P s^{-1}(\support(P s^{-1})) = 1. $$

(3) Suppose $\lambda$ is absolutely continuous with respect to Lebesgue measure, $\X$ is open and connected, and $s$ is real analytic. Let $L\subset\R^k$ be a hyperplane, and write $L =\{z\in\R^k: z^\t y = b\}$ where $y\in\R^k\less\{0\}$, $b\in\R$. Define $g:\X\to\R$ by $g(x) = s(x)^\t y-b$. Then $g$ is real analytic on $\X$, since a finite sum of real analytic functions is real analytic. Since $\X$ is connected, it follows that either $g$ is identically zero, or the set $V =\{x\in\X:g(x) = 0\}$ has Lebesgue measure zero \citep{Krantz_1992}. Now, $g$ cannot be identically zero, since for any $\theta\in\Theta$, letting $Z\sim P_\theta$, we have
$$ 0 < y^\t\kappa''(\theta)y = y^\t (\Cov s(Z)) y =\Var (y^\t s(Z)) =\Var g(Z) $$
by Proposition B.1(2) and (3). Consequently, $V$ must have Lebesgue measure zero. Hence, $P(V) = 0$, since $P$ is absolutely continuous with respect to $\lambda$, and thus, with respect to Lebesgue measure. Therefore,
$$\P(s(X)\in L) =\P(g(X) = 0) = P(V) = 0. \qedhere$$
\end{proof}

\begin{proof}[Proof of Lemma C.1]
By Taylor's theorem, for any $x\in B_\epsilon(x_0)$ there exists $z_x$ on the line between $x_0$ and $x$ such that, letting $y = x-x_0$,
$$ f(x) = f(x_0) + y^\t f'(x_0) +\tfrac{1}{2} y^\t f''(z_x) y =  f(x_0) +\tfrac{1}{2} y^\t f''(z_x) y. $$
Since $z_x\in B_\epsilon(x_0)$, and thus $A\before f''(z_x)\before B$,
$$\tfrac{1}{2} y^\t A y\leq f(x)-f(x_0)\leq\tfrac{1}{2} y^\t B y.$$
Hence,
\begin{align*}
e^{t f(x_0)}\int_{B_\epsilon(x_0)} \exp(-t f(x))\, dx
&\leq\int_{B_\epsilon(x_0)} \exp(-\tfrac{1}{2} (x-x_0)^\t (t A) (x-x_0)) \,dx \\
&=(2\pi)^{k/2} |(t A)^{-1}|^{1/2} \, \P\big(|(t A)^{-1/2} Z|\leq\epsilon\big).
\end{align*}
Along with a similar argument for the lower bound, this implies
$$\Big(\frac{2\pi}{t}\Big)^{k/2} \frac{C(t,\epsilon,B)}{|B|^{1/2}}
\leq e^{t f(x_0)} \int_{B_\epsilon(x_0)} \exp(-t f(x)) \,dx \leq
\Big(\frac{2\pi}{t}\Big)^{k/2} \frac{C(t,\epsilon,A)}{|A|^{1/2}}. $$
Considering the rest of the integral, outside of $B_\epsilon(x_0)$, we have
$$0\leq \int_{E\less B_\epsilon(x_0)} \exp(-t f(x)) \,dx 
\leq \exp\big(-(t-s)(f(x_0) +\d)\big)\,g(s).$$
Combining the preceding four inequalities yields the result.
\end{proof}

Although we do not need it (and thus, we omit the proof), the following corollary gives the well-known asymptotic form of the Laplace approximation. (As usual, $g(t)\sim h(t)$ as $t\to\infty$ means that $g(t)/h(t)\to 1$.)

\begin{corollary}
\label{corollary:Laplace-approximation}
Let $E\subset\R^k$ be open. Let $f: E\to\R$ be $C^2$ smooth such that for some $x_0\in E$ we have that $f'(x_0) = 0$, $f''(x_0)$ is positive definite, and $f(x)>f(x_0)$ for all $x\in E\less\{x_0\}$. Suppose there exists $\epsilon>0$ such that $B_\epsilon(x_0)\subset E$ and $\inf\{f(x)-f(x_0): x\in E\less B_\epsilon(x_0)\}$ is positive, and suppose there is some $s>0$ such that $\int_E e^{-s f(x)}\,dx<\infty$.  Then
$$\int_E \exp(-t f(x))\,dx \,\,\sim\,\, \Big(\frac{2\pi}{t}\Big)^{k/2}\frac{\exp(-t f(x_0))}{|f''(x_0)|^{1/2}} $$
as $t\to\infty$.
\end{corollary}

\section{Capture lemma}
\label{section:capture}

In this section, we prove Lemma 8.1, which is restated here for the reader's convenience.  

The following definitions are standard.
Let $\S$ denote the unit sphere in $\R^k$, that is, $\S =\{x\in\R^k:|x|= 1\}$. We say that $H\subset\R^k$ is a \emph{halfspace} if $H =\{x\in\R^k: x^\t u\prec b\}$, where $\prec$ is either $<$ or $\leq$, for some $u\in\S$, $b\in\R$.  We say that $L\subset\R^k$ is a \emph{hyperplane} if $L =\{x\in\R^k: x^\t u = b\}$ for some $u\in\S$, $b\in\R$. Given $U\subset\R^k$, let $\partial U$ denote the \emph{boundary of} $U$, that is, $\partial U =\bar U\less U^\circ$. So, for example, if $H$ is a halfspace, then $\partial H$ is a hyperplane. 
The following notation is also useful: given $x\in\R^k$, we call the set $R_x =\{a x: a>0\}$ the \emph{ray through $x$}. 

We give the central part of the proof first, postponing some plausible intermediate results for the moment.

\begin{lemma}[Capture lemma]
\label{lemma:capture-restated}
Let $V\subset\R^k$ be open and convex. Let $P$ be a probability measure on $\R^k$ such that:
\begin{enumerate}
\item $\E|X|<\infty$ when $X\sim P$,
\item $P(\bar V) = 1$, and
\item $P(L) = 0$ for any hyperplane $L$ that does not intersect $V$.
\end{enumerate}
If $X_1,X_2,\dotsc\iid P$, then for any $\beta\in(0,1]$ there exists $U\subset V$ compact such that $\Capture_\beta(X_{1:n},U)\xrightarrow[]{\mathrm{a.s.}} 1$ as $n\to\infty$.
\end{lemma}
\begin{proof}
Without loss of generality, we may assume $0\in V$ (since otherwise we can translate to make it so, obtain $U$, and translate back). Let $\beta\in(0,1]$. By Proposition \ref{proposition:halfspace} below, for each $u\in\S$ there is a closed halfspace $H_u$ such that $0\in H_u^\circ$, $R_u$ intersects $V\cap \partial H_u$, and $\Capture_\beta(X_{1:n},H_u)\xrightarrow[]{\mathrm{a.s.}} 1$ as $n\to\infty$. By Proposition \ref{proposition:intersection} below, there exist $u_1,\dotsc,u_r\in\S$ (for some $r>0$) such that the set $U =\I_{i = 1}^r H_{u_i}$ is compact and $U\subset V$. Finally,
$$\Capture_\beta(X_{1:n},U) = \prod_{i = 1}^r \Capture_\beta(X_{1:n},H_{u_i}) \xrightarrow[n\to\infty]{\mathrm{a.s.}} 1. \qedhere$$
\end{proof}

The main idea of the lemma is exhibited in the following simpler case, which we will use to prove Proposition \ref{proposition:halfspace}.

\begin{proposition}
\label{proposition:one-dimensional-case}
Let $V = (-\infty,c)$, where $-\infty<c\leq\infty$. Let $P$ be a probability measure on $\R$ such that:
\begin{enumerate}
\item $\E|X|<\infty$ when $X\sim P$, and
\item $P(V) = 1$.
\end{enumerate}
If $X_1,X_2,\dotsc\iid P$, then for any $\beta\in(0,1]$ there exists $b<c$ such that $\Capture_\beta(X_{1:n},(-\infty,b])\xrightarrow[]{\mathrm{a.s.}} 1$ as $n\to\infty$.
\end{proposition}
\begin{proof}
Let $\beta\in(0,1]$. By continuity from above, there exists $a<c$ such that $\P(X>a)<\beta$. If $\P(X>a) = 0$ then the result is trivial, taking $b = a$. Suppose $\P(X>a)>0$. Let $b$ such that $\E(X\mid X > a)<b<c$, which is always possible, by a straightforward argument (using $\E|X|<\infty$ in the $c=\infty$ case). Let $B_n = B_n(X_1,\dotsc,X_n) =\{i\in\{1,\dotsc,n\}: X_i>a\}$. Then
\begin{align*}
\frac{1}{|B_n|}\sum_{i\in B_n} X_i &=\frac{1}{\frac{1}{n}|B_n|}\frac{1}{n}\sum_{i = 1}^n X_i \,I(X_i>a) \\
&\xrightarrow[n\to\infty]{\textup{a.s.}} \frac{\E(X\,I(X>a))}{\P(X>a)} =\E(X\mid X>a)<b.
\end{align*}
Now, fix $n\in\{1,2,\dotsc\}$, and suppose $0<|B_n|<\beta n$ and $\frac{1}{|B_n|}\sum_{i\in B_n} X_i<b$, noting that with probability $1$, this happens for all $n$ sufficiently large. We show that this implies $\Capture_\beta(X_{1:n},(-\infty,b])= 1$. This will prove the result. 

Let $A\subset\{1,\dotsc,n\}$ such that $|A|\geq\beta n$. Let $M =\{\pi_1,\dotsc,\pi_{|A|}\}$ where $\pi$ is a permutation of $\{1,\dotsc,n\}$ such that $X_{\pi_1}\geq\cdots\geq X_{\pi_n}$ (that is, $M\subset\{1,\dotsc, n\}$ consists of the indices of $|A|$ of the largest entries of $(X_1,\dotsc,X_n)$).  Then $|M|=|A|\geq\beta n\geq|B_n|$, and it follows that $B_n\subset M$. Therefore,
$$\frac{1}{|A|}\sum_{i\in A} X_i\leq\frac{1}{|M|}\sum_{i\in M} X_i \leq \frac{1}{|B_n|}\sum_{i\in  B_n} X_i\leq b, $$
as desired.
\end{proof}

The first of the two propositions used in Lemma \ref{lemma:capture-restated} is the following.

\begin{proposition}
\label{proposition:halfspace}
Let $V$ and $P$ satisfy the conditions of Lemma \ref{lemma:capture-restated}, and also assume $0\in V$. If $X_1,X_2,\dotsc\iid P$ then for any $\beta\in(0,1]$ and any $u\in\S$ there is a closed halfspace $H\subset\R^k$ such that
\begin{enumerate}
\item $0\in H^\circ$,
\item $R_u$ intersects $V\cap \partial H$, and
\item $\Capture_\beta(X_{1:n},H)\xrightarrow[]{\mathrm{a.s.}} 1$ as $n\to\infty$.
\end{enumerate}
\end{proposition}
\begin{proof}
Let $\beta\in(0,1]$ and $u\in\S$. Either (a) $R_u\subset V$, or (b) $R_u$ intersects $\partial V$.

(Case (a)) Suppose $R_u\subset V$. Let $Y_i = X_i^\t u$ for $i = 1,2,\dotsc$. Then $\E|Y_i|\leq\E|X_i||u|=\E|X_i|<\infty$, and thus, by Proposition \ref{proposition:one-dimensional-case} (with $c =\infty$) there exists $b\in\R$ such that $\Capture_\beta(Y_{1:n},(-\infty,b])\xrightarrow[]{\mathrm{a.s.}} 1$. Let us choose this $b$ to be positive, which is always possible since $\Capture_\beta(Y_{1:n},(-\infty,b])$ is nondecreasing as a function of $b$. Let $H =\{x\in\R^k : x^\t u\leq b\}$. Then $0\in H^\circ$, since $b>0$, and $R_u$ intersects $V\cap \partial H$ at $b u$, since $R_u\subset V$ and $b u^\t u = b$. And since $\frac{1}{|A|}\sum_{i\in A} Y_i\leq b$ if and only if $\frac{1}{|A|}\sum_{i\in A} X_i \in H$, we have $\Capture_\beta(X_{1:n},H)\xrightarrow[]{\mathrm{a.s.}} 1$.

(Case (b)) Suppose $R_u$ intersects $\partial V$ at some point $z\in\R^k$. Note that $z\neq 0$ since $0\not\in R_u$. Since $\bar V$ is convex, it has a supporting hyperplane at $z$, and thus, there exist $v\in\S$ and $c\in\R$ such that $G =\{x\in\R^k: x^\t v\leq c\}$ satisfies $\bar V\subset G$  and $z\in\partial G$ \citep[11.2]{Rockafellar_1970}. Note that $c>0$ and $V\cap \partial G =\emptyset$ since $0\in V$ and $V$ is open. Letting $Y_i = X_i^\t v$ for $i = 1,2,\dotsc$, we have
$$\P(Y_i\leq c) =\P(X_i^\t v\leq c) =\P(X_i\in G)\geq\P(X_i\in\bar V) = P(\bar V) = 1,$$
and hence,
$$\P(Y_i\geq c) =\P(Y_i = c) = \P(X_i^\t v= c) =\P(X_i\in \partial G)= P(\partial G) = 0,$$
by our assumptions on $P$, since $\partial G$ is a hyperplane that does not intersect $V$.
Consequently, $\P(Y_i<c) = 1$. Also, as before, $\E|Y_i|<\infty$. Thus, by Proposition~\ref{proposition:one-dimensional-case}, there exists $b<c$ such that $\Capture_\beta(Y_{1:n},(-\infty,b])\xrightarrow[]{\mathrm{a.s.}} 1$. Since $c>0$, we may choose this $b$ to be positive (as before).  Letting $H =\{x\in\R^k: x^\t v\leq b\}$, we have $\Capture_\beta(X_{1:n},H)\xrightarrow[]{\mathrm{a.s.}} 1$.  Also, $0\in H^\circ$ since $b>0$.

Now, we must show that $R_u$ intersects $V\cap \partial H$. First, since $z\in R_u$ means $z = a u$ for some $a>0$, and since $z\in\partial G$ means $z^\t v = c>0$, we find that $u^\t v>0$ and $z = c u/u^\t v$. Therefore, letting $y = b u/u^\t v$, we have $y\in R_u\cap V\cap \partial H$, since
\begin{enumerate}
\item[(i)] $b/u^\t v>0$, and thus $y\in R_u$,
\item[(ii)] $y^\t v = b$, and thus $y\in\partial H$,
\item[(iii)] $0<b/u^\t v<c/u^\t v$, and thus $y$ is a (strict) convex combination of $0\in V$ and $z\in\bar V$, hence $y\in V$ \citep[6.1]{Rockafellar_1970}.\qedhere
\end{enumerate}
\end{proof}

To prove Proposition \ref{proposition:intersection}, we need the following geometrically intuitive facts.

\begin{proposition}
\label{proposition:open-projection}
Let $V\subset\R^k$ be open and convex, with $0\in V$. Let $H$ be a closed halfspace such that $0\in H^\circ$. Let $T =\{x/|x|: x\in V\cap \partial H\}$. Then
\begin{enumerate}
\item $T$ is open in $\S$,
\item $T =\{u\in\S: R_u\mbox{ intersects } V\cap \partial H\}$, and
\item if $x\in H$, $x\neq 0$, and $x/|x|\in T$, then $x\in V$.
\end{enumerate}
\end{proposition}
\begin{proof}
Write $H =\{x\in\R^k: x^\t v\leq b\}$, with $v\in\S$, $b>0$. Let $\S_+ =\{u\in\S: u^\t v>0\}$.
(1) Define $f:\partial H\to\S_+$ by $f(x) = x/|x|$, noting that $0\not\in\partial H$. It is easy to see that $f$ is a homeomorphism. Since $V$ is open in $\R^k$, then $V\cap\partial H$ is open in $\partial H$. Hence, $T = f(V\cap\partial H)$ is open in $\S_+$, and since $\S_+$ is open in $\S$, then $T$ is also open in $\S$. Items (2) and (3) are easily checked.
\end{proof}

\begin{proposition}
\label{proposition:intersection}
Let $V\subset\R^k$ be open and convex, with $0\in V$. If $(H_u: u\in\S)$ is a collection of closed halfspaces such that for all $u\in\S$,
\begin{enumerate}
\item $0\in H_u^\circ$ and
\item $R_u$ intersects $V\cap\partial H_u$,
\end{enumerate}
then there exist $u_1,\dotsc,u_r\in\S$ (for some $r>0$) such that the set $U =\I_{i = 1}^r H_{u_i}$ is compact and $U\subset V$.
\end{proposition}
\begin{proof}
For $u\in\S$, define $T_u =\{x/|x|: x\in V\cap\partial H_u\}$. By part (1) of Proposition \ref{proposition:open-projection}, $T_u$ is open in $\S$, and by part (2), $u\in T_u$, since $R_u$ intersects $V\cap\partial H_u$. Thus, $(T_u: u\in\S)$ is an open cover of $\S$. Since $\S$ is compact, there is a finite subcover: there exist $u_1,\dotsc,u_r\in\S$ (for some $r>0$) such that $\U_{i = 1}^r T_{u_i} \supset \S$, and in fact, $\U_{i = 1}^r T_{u_i} = \S$. Let $U =\I_{i = 1}^r H_{u_i}$. Then $U$ is closed and convex (as an intersection of closed, convex sets). Further, $U\subset V$ since for any $x\in U$, if $x = 0$ then $x\in V$ by assumption, while if $x\neq 0$ then $x/|x|\in T_{u_i}$ for some $i\in\{1,\dotsc,r\}$ and $x\in U\subset H_{u_i}$, so $x\in V$ by Proposition \ref{proposition:open-projection}(3).

In order to show that $U$ is compact, we just need to show it is bounded, since we already know it is closed. Suppose not, and let $x_1,x_2,\dotsc\in U\less\{0\}$ such that $|x_n|\to\infty$ as $n\to\infty$. Let $v_n = x_n/|x_n|$. Since $\S$ is compact, then $(v_n)$ has a convergent subsequence such that $v_{n_i}\to u$ for some $u\in\S$. Then for any $a>0$, we have $a v_{n_i}\in U$ for all $i$ sufficiently large (since $a v_{n_i}$ is a convex combination of $0\in U$  and $|x_{n_i}| v_{n_i} = x_{n_i}\in U$ whenever $|x_{n_i}|\geq a$). Since $a v_{n_i}\to a u$, and $U$ is closed, then $a u\in U$. Thus, $a u\in U$ for all $a>0$, i.e. $R_u\subset U$. But $u\in T_{u_j}$ for some $j\in\{1,\dotsc,r\}$, so $R_u$ intersects $\partial H_{u_j}$ (by Proposition \ref{proposition:open-projection}(2)), and thus $a u\not\in H_{u_j}\supset U$ for all $a>0$ sufficiently large. This is a contradiction. Therefore, $U$ is bounded.
\end{proof}

\bibliography{refs}
\bibliographystyle{imsart-nameyear}